\documentclass{gtpart}

\usepackage{amssymb}
\usepackage{amsthm}
\usepackage{amsmath}

\usepackage[normalem]{ulem}

\usepackage[usenames]{color}
\usepackage{xcolor}

\usepackage{tikz-cd}

\usepackage[shortlabels]{enumitem}

\usepackage{hyperref}

\theoremstyle{plain}
\newtheorem{proposition}{Proposition}[section]
\newtheorem{theorem}[proposition]{Theorem}
\newtheorem{lemma}[proposition]{Lemma}
\newtheorem{corollary}[proposition]{Corollary}
\newtheorem{observation}[proposition]{Observation}
\theoremstyle{definition}

\newtheorem{definition}[proposition]{Definition}
\theoremstyle{remark}
\newtheorem{remark}[proposition]{Remark}

\DeclareMathOperator{\Aut}{Aut}

\DeclareMathOperator{\SL}{SL}
\DeclareMathOperator{\GL}{GL}
\DeclareMathOperator{\SO}{SO}

\DeclareMathOperator{\PSL}{PSL}
\DeclareMathOperator{\PGL}{PGL}

\DeclareMathOperator{\End}{End}

\DeclareMathOperator{\image}{image} 
 
\DeclareMathOperator{\Spanset}{Span} 
\DeclareMathOperator{\Gr}{Gr} 
\DeclareMathOperator{\id}{id} 
 
\DeclareMathOperator{\CAT}{CAT}

\DeclareMathOperator{\Min}{Min}

\DeclareMathOperator{\relint}{rel-int}

\DeclareMathOperator{\Ec}{\mathcal{E}}

\DeclareMathOperator{\Lc}{\mathcal{L}}

\DeclareMathOperator{\Oc}{\mathcal{O}}

\DeclareMathOperator{\Ab}{\mathbb{A}}

\DeclareMathOperator{\Cb}{\mathbb{C}}

\DeclareMathOperator{\Hb}{\mathbb{H}}

\DeclareMathOperator{\Pb}{\mathbb{P}}
\DeclareMathOperator{\Rb}{\mathbb{R}}

\DeclareMathOperator{\Zb}{\mathbb{Z}}

\DeclareMathOperator{\gL}{\mathfrak{g}}

\newcommand{\abs}[1]{\left|#1\right|}

\newcommand{\norm}[1]{\left\|#1\right\|}

\begin{document}

\title[A higher rank rigidity theorem]{A higher rank rigidity theorem for convex real projective manifolds}

\author{Andrew Zimmer}
\address{Department of Mathematics, Louisiana State University, Baton Rouge, LA 70803 \newline Current Address: Department of Mathematics, University of Wisconsin-Madison, Madison, WI, 53706}
\email{amzimmer2@wisc.edu}
\date{\today}

\begin{abstract} For convex real projective manifolds we prove an analogue of the higher rank rigidity theorem of Ballmann and Burns-Spatzier.

\end{abstract}

\maketitle

\section{Introduction}

A real projective structure on a $d$-manifold $M$ is an open cover $M = \cup_{\alpha} U_\alpha$ along with coordinate charts $\varphi_\alpha : U_{\alpha} \rightarrow \Pb(\Rb^{d+1})$  such that each transition function $\varphi_{\alpha} \circ \varphi_{\beta}^{-1}$ coincides with the restriction of an element in $\PGL_{d+1}(\Rb)$. A \emph{real projective manifold} is a manifold equipped with a real projective structure. 

An important class of real projective manifolds are the convex real projective manifolds, which are defined as follows. First, a subset $\Omega \subset \Pb(\Rb^{d+1})$ is called a \emph{properly convex domain} if there exists an affine chart which contains it as a bounded convex open set. In this case, the \emph{automorphism group of $\Omega$} is  
\begin{align*}
\Aut(\Omega) := \{ g \in \PGL_{d+1}(\Rb): g \Omega = \Omega\}.
\end{align*}
If $\Gamma\leq \Aut(\Omega)$ is a discrete subgroup that acts freely and properly discontinuously on $\Omega$, then the quotient manifold $\Gamma \backslash \Omega$ is called a \emph{convex real projective manifold}. Notice that local inverses to the covering map $\Omega \rightarrow \Gamma \backslash \Omega$ provide a real projective structure on the quotient. In the case when there exists a compact quotient the domain $\Omega$ is called \emph{divisible}. For more background see the expository papers~\cite{B2008,Q2010,L2014}.

When $d \leq 3$, the structure of closed convex real projective $d$-manifolds are very well understood thanks to deep work of Benz\'{e}cri~\cite{B1960}, Goldman~\cite{G1990}, and Benoist~\cite{B2006}. But when $d \geq 4$ their general structure is mysterious. 

In this paper we establish a dichotomy for convex real projective manifolds inspired by the theory of non-positively curved Riemannian manifolds. In particular, a compact Riemannian manifold $(M,g)$ with non-positive curvature is said to have \emph{higher rank} if every geodesic in the universal cover is contained in a totally geodesic subspace isometric to $\Rb^2$. Otherwise $(M,g)$ is said to have \emph{rank one}.  An important theorem of Ballmann~\cite{Ballmann1985} and Burns-Spatzier~\cite{BS1987a,BS1987b} states that every compact irreducible Riemannian manifold  with non-positive curvature and higher rank is a locally symmetric space. This foundational result reduces many problems about non-positively curved manifolds to the rank one case. Further, rank one manifolds possess very useful ``weakly hyperbolic behavior'' (see for instance~\cite{B1982,K1998}).

In the context of convex real projective manifolds, the natural analogue of isometrically embedded copies of $\Rb^2$ are properly embedded simplices, see Section~\ref{sec:PES} below, which leads to the following definition of higher rank. 

\begin{definition} \ \begin{enumerate}
\item A properly convex domain $\Omega \subset \Pb(\Rb^d)$ has \emph{higher rank} if for every $p,q \in \Omega$ there exists a properly embedded simplex $S \subset \Omega$ with $\dim(S) \geq 2$ and $[p,q] \subset S$. 
\item If a properly convex domain $\Omega \subset \Pb(\Rb^d)$ does not have higher rank, then we say that $\Omega$ has \emph{rank one}. 
\end{enumerate}
\end{definition}

There are two basic families of properly convex domains with higher rank: reducible domains (see Section~\ref{sec:convex_and_irreducible} below) and symmetric domains with real rank at least two. 

A properly convex domain $\Omega \subset \Pb(\Rb^d)$ is called \emph{symmetric} if there exists a semisimple Lie group $G \leq \PGL_d(\Rb)$ which preserves $\Omega$ and acts transitively. In this case, the \emph{real rank of $\Omega$} is defined to be the real rank of $G$. Koecher and Vinberg characterized the irreducible symmetric properly convex domains and proved that $G$ must be locally isomorphic to either 
\begin{enumerate}
\item $\SO(1,m)$ with $d =m+1$, 
\item $\SL_m(\Rb)$ with $d=(m^2+m)/2$, 
\item $\SL_m(\Cb)$ with $d=m^2$,  
\item $\SL_m(\Hb)$ with $d = 2m^2-m$, or 
\item $E_{6(-26)}$ with $d=27$. 
\end{enumerate}
For details see~\cite{K1999, V1963, V1965,FK1994}. Borel~\cite{B1963} proved that every semisimple  Lie group contains a co-compact lattice, which implies that every symmetric properly convex domain is divisible. 

In this paper we prove that these two families of examples are the only divisible domains with higher rank. In fact, we show that being symmetric with real rank at least two is equivalent to a number of other ``higher rank'' conditions. Before stating the main result we need a few more definitions.

\begin{definition} \ 
\begin{itemize}
\item Given $g \in \PGL_d(\Rb)$ let 
\begin{align*}
\lambda_1(g) \geq \lambda_2(g) \geq \dots \geq \lambda_d(g)
\end{align*}
denote the absolute values of the eigenvalues of some (hence any) lift of $g$ to $\SL_d^{\pm}(\Rb):=\{ h \in \GL_d(\Rb) : \det h = \pm 1\}$. 
\item  $g \in \PGL_d(\Rb)$ is \emph{proximal} if $\lambda_1(g) > \lambda_2(g)$. In this case, let $\ell_g^+ \in \Pb(\Rb^d)$ denote the eigenline of $g$ corresponding to $\lambda_1(g)$.
\item $g \in \PGL_d(\Rb)$ is \emph{bi-proximal} if $g, g^{-1}$ are both proximal. In this case, define $\ell^-_g : = \ell^+_{g^{-1}}$.
\end{itemize}
\end{definition}

Next we define a distance on the boundary using projective line segments. 

\begin{definition}
Given a properly convex domain $\Omega \subset\Pb(\Rb^d)$  the (possibly infinite valued) \emph{simplicial distance} on $\partial \Omega$ is defined by
\begin{align*}
s_{\partial\Omega}(x,y) = \inf\{ k : & \ \exists \ a_0,\dots, a_{k} \text{ with } x=a_0, \ y =a_{k}, \text{ and } \\
& [a_j,a_{j+1}] \subset \partial \Omega \text{ for all } 0\leq j \leq k-1\}.
\end{align*}
\end{definition}

We will prove the following characterization of higher rank in the context of convex real projective manifolds. 

\begin{theorem}\label{thm:characterization}(see Section~\ref{sec:pf_of_char}) Suppose  that $\Omega \subset \Pb(\Rb^d)$ is an irreducible properly convex domain and $\Gamma \leq \Aut(\Omega)$ is a discrete group acting co-compactly on $\Omega$. Then the following are equivalent: 
\begin{enumerate}
\item $\Omega$ is symmetric with real rank at least two, 
\item $\Omega$ has higher rank,
\item \label{char:line_between_ext_pts} the extreme points of $\Omega$ form a closed proper subset of $\partial \Omega$,
\item $[x_1,x_2] \subset \partial \Omega$ for every two extreme points $x_1,x_2 \in \partial \Omega$,
\item $s_{\partial\Omega}(x,y) \leq 2$ for all $x,y \in \partial \Omega$,
\item $s_{\partial\Omega}(x,y) < +\infty$ for all $x,y \in \partial \Omega$,
\item $\Gamma$ has higher rank in the sense of Prasad-Raghunathan (see Section~\ref{sec:rank_PS}), 
\item for every $g \in \Gamma$ with infinite order the cyclic group $g^{\Zb}$ has infinite index in the centralizer of $g$ in $\Gamma$,
\item every $g \in \Gamma$ with infinite order has at least three fixed points in $\partial \Omega$, 
\item\label{char:line_between_attracting_repelling} $[\ell_g^+, \ell_g^-] \subset \partial \Omega$ for every bi-proximal element $g \in \Gamma$, and
\item $s_{\partial\Omega}(\ell_g^+, \ell_g^-) < +\infty$ for every bi-proximal element $g \in \Gamma$.
\end{enumerate}
\end{theorem}

M. Islam~\cite{I2019} has recently defined and studied rank one isometries of a properly convex domain. These are analogous to the classical definition of rank one isometries of $\CAT(0)$ spaces (see~\cite{B1982}) and are defined as follows. 

\begin{definition}[M. Islam~\cite{I2019}] Suppose that $\Omega \subset \Pb(\Rb^d)$ is a properly convex domain. An element $g \in \Aut(\Omega)$ is a \emph{rank one  isometry} if $g$ is bi-proximal and $s_{\partial \Omega}(\ell_g^+,\ell_g^-) > 2$. 
\end{definition}

\begin{remark} \ 
\begin{enumerate}
\item When $g \in \Aut(\Omega)$ is a rank one isometry, then the properly embedded line segment $(\ell^+_g, \ell^-_g) \subset \Omega$ is preserved by $g$. Further, $g$ acts by translations on $(\ell^+_g, \ell^-_g)$ in the following sense: if $H_\Omega$ is the Hilbert metric on $\Omega$, then there exists $T > 0$ such that 
\begin{align*}
H_\Omega(g^n(x), x) = nT
\end{align*}
for all $n \geq 0$ and $x \in (\ell^+_g, \ell^-_g)$. 
\item M. Islam~\cite[Proposition 6.3]{I2019} also proved the following weaker characterization of rank one isometries: $g \in \Aut(\Omega)$ is a rank one isometry if and only if  $g$ acts by translations on a properly embedded line segment $(a,b) \subset \Omega$ and $s_{\partial \Omega}(a, b) > 2$. 
\end{enumerate}
 \end{remark}

As an immediate consequence of Theorem~\ref{thm:characterization} we obtain the following. 

\begin{corollary} Suppose that $\Omega \subset \Pb(\Rb^d)$ is an irreducible properly convex domain and $\Gamma \leq \Aut(\Omega)$ is a discrete group acting co-compactly on $\Omega$. Then the following are equivalent: 
\begin{enumerate}
\item $\Omega$ has rank one, 
\item $\Gamma$ contains a rank one  isometry.
\end{enumerate}
\end{corollary}

M. Islam has also established a number of remarkable results when the automorphism group contains a rank one isometry, see~\cite{I2019} for details. For instance combining Theorem~\ref{thm:characterization} with~\cite[Theorem 1.5]{I2019} yields:

\begin{corollary}[{Consequence of Theorem~\ref{thm:characterization} and~\cite[Theorem 1.5]{I2019}}] Suppose that $\Omega \subset \Pb(\Rb^d)$ is an irreducible properly convex domain and $\Gamma \leq \Aut(\Omega)$ is a discrete group acting co-compactly on $\Omega$. If $d \geq 3$ and $\Omega$ is not symmetric with real rank at least two, then $\Gamma$ is an acylindrically hyperbolic group.
\end{corollary}

\subsection{Outline of the proof of Theorem~\ref{thm:characterization}} The difficult part is showing that any one of the conditions (2) through (11) implies that the domain is symmetric with real rank at least two.

One key idea is to construct and study special semigroups in $\Pb(\End(\Rb^d))$ associated to each boundary face. This is accomplished as follows. First, motivated by a lemma of Benoist~\cite[Lemma 2.2]{B2003b}, we consider the following compactification of a subgroup of $\PGL_d(\Rb)$. 

\begin{definition} Given a subgroup $G \leq \PGL_d(\Rb)$ let 
\begin{align*}
\overline{G}^{\End} \subset \Pb(\End(\Rb^d))
\end{align*}
denote the closure of $G$ in $\Pb(\End(\Rb^d))$. \end{definition}

Next for a dividing group we introduce the following subsets of this compactification.

\begin{definition}
Suppose  that $\Omega \subset \Pb(\Rb^d)$ is a properly convex domain and $\Gamma \leq \Aut(\Omega)$ is a discrete group acting co-compactly on $\Omega$. If  $F \subset \partial\Omega$ is a boundary face and  $V := \Spanset F \subset \Rb^d$, then define 
\begin{align*}
\overline{\Gamma}^{\End}_F :=\left\{  T \in \overline{\Gamma}^{\End}: {\rm image}(T) \subset V\right\}
\end{align*}
and 
 \begin{align*}
\overline{\Gamma}^{\End}_{F,\star} :=\left\{  T \in \overline{\Gamma}^{\End}: {\rm image}(T) = V \text{ and } \ker(T) \cap V = \{0\}\right\}.
\end{align*}
\end{definition}

We then prove the following result about these subsets. 

\begin{theorem}\label{thm:structure_subsets_intro}(see Theorem~\ref{thm:structure_subsets} below) Suppose that $\Omega \subset \Pb(\Rb^d)$ is an irreducible properly convex domain and $\Gamma \leq \Aut(\Omega)$ is a discrete group acting co-compactly on $\Omega$. If $\Omega$ is non-symmetric, $F \subset \partial\Omega$ is a boundary face, $V:=\Spanset F \subset \Rb^d$, and $\dim(V) \geq 2$, then: 
\begin{enumerate}[(a)]
\item If $T \in \overline{\Gamma}^{\End}_F$, then $T(\Omega) \subset \overline{F}$.
\item If $T \in \overline{\Gamma}^{\End}_{F,\star}$, then $T(F)$ is an open subset of $F$. 
\item The set 
\begin{align*}
\left\{ T|_V : T \in \overline{\Gamma}^{\End}_{F,\star}\right\}
\end{align*}
is a non-discrete Zariski dense semigroup in $\Pb(\End(V))$.
\end{enumerate}
\end{theorem}

Using Theorem~\ref{thm:structure_subsets_intro} we will show that any one of the conditions (2) through (11) in Theorem~\ref{thm:characterization} implies that the domain is symmetric with real rank at least two. Here is a sketch of the argument, first suppose that $\Omega \subset \Pb(\Rb^d)$ is an irreducible properly convex domain, $\Gamma \leq \Aut(\Omega)$ is a discrete group acting co-compactly on $\Omega$, and any one of the conditions (2) through (11) in Theorem~\ref{thm:characterization} is true. Then let $\Ec_\Omega \subset \partial \Omega$ denote the extreme points of $\Omega$. We will show that there exists a boundary face $F \subset \partial\Omega$ such that 
\begin{align} 
\label{eq:intersection_i}
F \cap \overline{\Ec_\Omega} = \emptyset.
\end{align}
By choosing $F$ minimally, we can also assume that $\overline{\Ec_\Omega}$ intersects every boundary face of strictly smaller dimension. As before let $V : =\Spanset F$. Then using Equation~\eqref{eq:intersection_i} we show that $T|_V \in \Aut(F)$ for every $T \in \overline{\Gamma}^{\End}_{F,\star}$. Then Theorem~\ref{thm:structure_subsets_intro} implies that either $\Omega$ is symmetric or $\Aut(F)$ is a non-discrete  Zariski dense subgroup of $\PGL(V)$. In the latter case, it is fairly easy to deduce that $\PSL(V) \subset \Aut(F)$, see Lemma~\ref{lem:big_Z_closure} below, which is impossible. So $\Omega$ must be symmetric. 

\subsection{Outline of the paper} In Section~\ref{sec:prelim} we recall some preliminary material. In Section~\ref{sec:semigroups_associated_to_face} we prove Theorem~\ref{thm:structure_subsets_intro}. In Section~\ref{sec:main_rigidity_result} we prove the rigidity result mentioned in the previous subsection. 

The rest of the paper is devoted to the proof of the various equivalences in Theorem~\ref{thm:characterization}. In Sections~\ref{sec:density_of_fixed_pts}, \ref{sec:NS_dynamics}, and~\ref{sec:Fixed points and centralizers} we prove some new results about the action of the automorphism group. In Section~\ref{sec:rank_PS} we consider the rank of a group in the sense of Prasad-Raghunathan. Finally, in Section~\ref{sec:pf_of_char} we prove Theorem~\ref{thm:characterization}.

 \subsection*{Acknowledgements} I would like to thank Ralf Spatzier and Mitul Islam for helpful conversations. I would also like to thank the referee for their very useful corrections and suggestions. This material is based upon work supported by the National Science Foundation under grants DMS-1904099, DMS-2105580, and DMS-2104381.
 
\section{Preliminaries}\label{sec:prelim}

\subsection{Notations} Given a linear subspace $V \subset \Rb^d$, we let $\Pb(V) \subset \Pb(\Rb^d)$ denote its projectivization. In all other cases, given some object $o$ we will let $[o]$ be the projective equivalence class of $o$, for instance: 
\begin{enumerate}
\item if $v \in \Rb^{d} \setminus \{0\}$ let $[v]$ denote the image of $v$ in $\Pb(\Rb^{d})$, 
\item if $\phi \in \GL_{d}(\Rb)$ let $[\phi]$ denote the image of $\phi$ in $\PGL_{d}(\Rb)$, and 
\item if $T \in \End(\Rb^{d}) \setminus\{0\}$ let $[T]$ denote the image of $T$ in $\Pb(\End(\Rb^{d}))$. 
\end{enumerate}
We also identify $\Pb(\Rb^d) = \Gr_1(\Rb^d)$, so for instance: if $x \in \Pb(\Rb^d)$ and $V \subset \Rb^d$ is a linear subspace, then $x \in \Pb(V)$ if and only if $x \subset V$. 

Finally, given a subset $X$ of $\Rb^d$ (respectively $\Pb(\Rb^d)$) we will let $\Spanset X \subset \Rb^d$ denote the smallest linear subspace containing $X$ (respectively, the preimage of $X$). 

\subsection{Convexity and line segments} A subset $C \subset \Pb(\Rb^d)$ is called \emph{convex} if there exists an affine chart which contains it as a convex subset. A subset $C \subset \Pb(\Rb^d)$ is called \emph{properly convex} if there exists an affine chart which contains it as a bounded convex subset. For convex subsets, we make the following topological definitions.

\begin{definition}\label{defn:topology} Suppose that $C \subset \Pb(\Rb^d)$ is a convex set. The \emph{relative interior of $C$}, denoted by $\relint(C)$, is  the interior of $C$ in its span and the \emph{boundary of $C$} is $\partial C : = \overline{C} \setminus \relint(C)$.
\end{definition}

A \emph{line segment} in $\Pb(\Rb^{d})$ is a connected subset of a projective line. Given two points $x,y \in \Pb(\Rb^{d})$ there is no canonical line segment with endpoints $x$ and $y$, but we will use the following convention: if $C \subset \Pb(\Rb^d)$ is a properly convex set and $x,y \in \overline{C}$, then (when the context is clear) we will let $[x,y]$ denote the closed line segment joining $x$ to $y$ which is contained in $\overline{C}$. In this case, we will also let $(x,y)=[x,y]\setminus\{x,y\}$, $[x,y)=[x,y]\setminus\{y\}$, and $(x,y]=[x,y]\setminus\{x\}$.

\subsection{Irreduciblity} A subgroup $\Gamma \leq \PGL_d(\Rb)$ is \emph{irreducible} if $\{0\}$ and $\Rb^d$ are the only $\Gamma$-invariant linear subspaces of $\Rb^d$ and \emph{strongly irreducible} if every finite index subgroup is irreducible. 

We will use the following observation several times. 

\begin{observation}\label{obs:str_irred_transverse} If $\Gamma \leq \PGL_d(\Rb)$ is strongly irreducible, $x_1, \dots, x_k \in \Pb(\Rb^d)$, and 
$$
V_1, \dots, V_k \subsetneq \Rb^d
$$ 
are linear subspaces, then there exists $g \in \Gamma$ such that $gx_j \notin \Pb(V_j)$ for all $1 \leq j \leq k$. 
\end{observation}

\begin{proof} Let $G = \overline{\Gamma}^{Zar}$ denote the Zariski closure of $\Gamma$ in $ \PGL_d(\Rb)$ and let $G_0 \leq G$ denote the connected component of the identity of $G$ (in the Zariski topology). Then, $G_0 \cap \Gamma$ is a finite index subgroup of $\Gamma$ and hence $G_0$ is irreducible. So each set 
$$
\Oc_j = \{ g \in G_0 : gx_j \notin \Pb(V_j) \}
$$
is non-empty and Zariski open in $G_0$. Hence $\Oc = \cap_{j=1}^k \Oc_j$ is non-empty and Zariski open in $G_0$. Since $\Gamma \cap G_0$ is Zariski dense in $G_0$, there exists some $g \in \Gamma \cap \Oc$. 
\end{proof}

\subsection{Zariski closures}\label{sec:convex_and_irreducible}

An open convex cone $C \subset \Rb^d$ is  \emph{reducible} if there exists a non-trivial vector space decomposition $\Rb^d = V_1 \oplus V_2$ and convex cones $C_1 \subset V_1$ and $C_2 \subset V_2$ such that $C=C_1 + C_2$. Otherwise, $C$ is said to be \emph{irreducible}. The preimage in $\Rb^d$ of a properly convex domain $\Omega \subset \Pb(\Rb^d)$ is the union of a cone and its negative, when this cone is reducible (respectively irreducible) we say that $\Omega$ is \emph{reducible} (respectively \emph{irreducible}). 

Benoist determined the Zariski closures of discrete groups acting co-compactly on irreducible properly convex domains.

\begin{theorem}[Benoist~\cite{B2003b}]\label{thm:Zariski_closure} Suppose that $\Omega \subset \Pb(\Rb^d)$ is an irreducible properly convex domain and $\Gamma \leq \Aut(\Omega)$ is a discrete group acting co-compactly on $\Omega$. Then either 
\begin{enumerate}
\item $\Omega$ is symmetric or 
\item $\Gamma$ is Zariski dense in $\PGL_d(\Rb)$. 
\end{enumerate}
\end{theorem}

%We will also need the following result of Gol'dsheid-Margulis.
%
%\begin{theorem}[Gol'dsheid-Margulis~\cite{GM1989}]\label{thm:GM} Suppose that $\Gamma \leq \PGL_d(\Rb)$ is a subgroup. If the Zariski closure of $\Gamma$ in $\PGL_d(\Rb)$ contains a bi-proximal element, then so does $\Gamma$. \end{theorem}

\subsection{The Hilbert distance} In this section we recall the definition of the Hilbert metric, but first some notation. 

Given a projective line $L \subset \Pb(\Rb^d)$ and four distinct points $a,x,y,b\in L$ we define the \emph{cross ratio} by 
\begin{align*}
[a,x,y,b] = \frac{\abs{x-b}\abs{y-a}}{\abs{x-a}\abs{y-b}}
\end{align*}
where $\abs{\cdot}$ is some (any) norm in some (any) affine chart of $\Pb(\Rb^d)$ containing $a,x,y,b$. 

Next, for $x,y \in \Pb(\Rb^d)$ distinct let $L_{x,y} \subset \Pb(\Rb^d)$ denote the projective line containing $x$ and $y$. 

\begin{definition} Suppose that $\Omega \subset \Pb(\Rb^d)$ is a properly convex domain. The \emph{Hilbert distance}, denoted by $H_\Omega$, on $\Omega$ is defined as follows: if $x,y \in \Omega$ are distinct, then 
\begin{align*}
H_\Omega(x,y) = \frac{1}{2} \log [a,x,y,b]
\end{align*}
where $\partial \Omega \cap L_{x,y} = \{a,b\}$ with the ordering $a,x,y,b$ along $L_{x,y}$. 
\end{definition}

The following result is classical (see for instance~\cite[Section 28]{BK1953}). 
 
 \begin{proposition}\label{prop:hilbert_basic}
Suppose that $\Omega \subset \Pb(\Rb^{d})$ is a properly convex domain. Then $H_{\Omega}$ is a complete $\Aut(\Omega)$-invariant metric on $\Omega$ which generates the standard topology on $\Omega$. Moreover, if $p,q \in \Omega$, then there exists a geodesic joining $p$ and $q$ whose image is the line segment $[p,q]$.
\end{proposition}

\subsection{Properly embedded simplices}\label{sec:PES} In this subsection we recall the definition of properly embedded simplices.

\begin{definition} A subset $S \subset \Pb(\Rb^d)$ is a \emph{simplex} if there exists $g \in \PGL_d(\Rb)$ and $k \geq 0$ such that 
\begin{align*}
gS = \left\{ [x_1:\dots:x_{k+1}:0:\dots:0] \in \Pb(\Rb^d): x_1>0,\dots, x_{k+1}>0  \right\}.
\end{align*}
In this case, we write $\dim(S) = k$ (notice that $S$ is homeomorphic to $\Rb^k$). 
\end{definition}

\begin{definition} Suppose that $A \subset B \subset \Pb(\Rb^d)$. Then $A$ is \emph{properly embedded in $B$} if the inclusion map $A \hookrightarrow B$ is a proper map (relative to the subspace topology). 
\end{definition} 

The Hilbert metric on a simplex is isometric to a normed space (see~\cite[Proposition 1.7]{N1988}, ~\cite{dlH1993}, or ~\cite{V2014}) and so we have the following observation. 

\begin{observation} Suppose that $\Omega \subset \Pb(\Rb^d)$ is a properly convex domain and $S \subset \Omega$ is a properly embedded simplex. Then $(S,H_\Omega)$ is quasi-isometric to $\Rb^{\dim S}$. 
\end{observation}

\subsection{Limits of linear maps} Every $T \in \Pb(\End(\Rb^d))$ induces a map 
\begin{align*}
\Pb(\Rb^d) \setminus \Pb(\ker T) \rightarrow \Pb(\Rb^d)
\end{align*}
defined by $x \rightarrow T(x)$. We will frequently use the following observation.

\begin{observation}\label{obs:limit of endomorphisms} If $(T_n)_{n \geq 1}$ is a sequence in $\Pb(\End(\Rb^d))$ converging to $T \in \Pb(\End(\Rb^d))$, then 
\begin{align*}
T(x) = \lim_{n \rightarrow \infty} T_n(x)
\end{align*}
for all $x \in \Pb(\Rb^d) \setminus \Pb(\ker T)$. Moreover, the convergence is uniform on compact subsets of $ \Pb(\Rb^d) \setminus \Pb(\ker T)$. 
\end{observation}

\subsection{The faces and extreme points of a properly convex domain} 

\begin{definition} Suppose that $\Omega \subset \Pb(\Rb^d)$ is a properly convex domain. For $x \in \overline{\Omega}$ let $F_\Omega(x)$ denote the \emph{(open) face} of $x$, that is 
\begin{align*}
F_\Omega(x) = \{ x\} \cup \left\{ y \in \overline{\Omega} : \text{ $\exists$ an open line segment in $\overline{\Omega}$ containing $x$ and $y$} \right\}.
\end{align*}
 If $x \in \partial \Omega$ and $F_\Omega(x)=\{x\}$, then $x$ is called an \emph{extreme point of $\Omega$}. Finally, let 
 \begin{align*}
 \Ec_\Omega \subset \partial\Omega
 \end{align*}
 denote the set of all extreme points. 
\end{definition}

These subsets have the following basic properties. 

\begin{observation}\label{obs:faces} Suppose that $\Omega \subset \Pb(\Rb^d)$ is a properly convex domain. 
\begin{enumerate}
\item If $x \in \Omega$, then $F_\Omega(x) = \Omega$. 
\item $F_\Omega(x)$ is open in its span.
\item $y \in F_\Omega(x)$ if and only if $x \in F_\Omega(y)$ if and only if $F_\Omega(x) = F_\Omega(y)$.
\item if $y \in \partial F_\Omega(x)$, then $F_\Omega(y) \subset \partial F_\Omega(x)$ and $F_{\Omega}(y) = F_{F_\Omega(x)}(y)$. 
\item If $x, y \in \overline{\Omega}$ and $z \in (x,y)$, then
\begin{align*}
(p,q) \subset F_\Omega(z)
\end{align*}
for all $p \in  F_\Omega(x)$ and $q \in F_\Omega(y)$. 
\end{enumerate}
\end{observation}

\begin{proof} These are all simple consequences of convexity. \end{proof}

We will also use the following results about the action of the automorphism group.

\begin{proposition}\cite[Proposition 5.6]{IZ2019}\label{prop:dynamics_of_automorphisms}
Suppose that $\Omega \subset \Pb(\Rb^d)$ is a properly convex domain, $p_0 \in \Omega$, and $(g_n)_{n \geq 1}$ is a sequence in $\Aut(\Omega)$ such that 
\begin{enumerate}
\item $g_n (p_0) \rightarrow x \in \partial \Omega$, 
\item $g_n^{-1} (p_0) \rightarrow y \in \partial \Omega$, and
\item $g_n$ converges in $\Pb(\End(\Rb^d))$ to $T \in \Pb(\End(\Rb^d))$. 
\end{enumerate}
Then $\image T \subset \Spanset  F_\Omega(x)$, $\Pb(\ker T) \cap \Omega = \emptyset$, and $y \in \Pb(\ker T)$. 
\end{proposition} 

In the case of ``non-tangential'' convergence we can say more. 

\begin{proposition}\cite[Proposition 5.7]{IZ2019}\label{prop:dynamics_of_automorphisms_2}
Suppose that $\Omega \subset \Pb(\Rb^d)$ is a properly convex domain, $p_0 \in \Omega$, $x \in \partial \Omega$, $(p_n)_{n \geq 1}$ is a sequence in $[p_0, x)$ converging to $x$, and $(g_n)_{n \geq 1}$ is a sequence in $\Aut(\Omega)$ such that 
\begin{align*}
\sup_{n \geq 1} H_\Omega(g_n (p_0), p_n) < + \infty.
\end{align*}
If $g_n$ converges in $\Pb(\End(\Rb^d))$ to $T \in \Pb(\End(\Rb^d))$, then
\begin{align*}
T(\Omega) = F_\Omega(x)
\end{align*}
and hence $\image T = \Spanset  F_\Omega(x)$. 
\end{proposition}

Proposition 5.7 in~\cite{IZ2019} is stated differently, so we provide the proof. 

\begin{proof} Proposition~\ref{prop:dynamics_of_automorphisms} implies that $T(\Omega) \subset F_\Omega(x)$ so we just have to prove that $T(\Omega) \supset F_\Omega(x)$. 

Fix $y \in F_\Omega(x)$. Then we can pick a sequence $(y_n)_{n \geq 1}$ in  $[p_0, y)$ such that 
\begin{align*}
\sup_{n \geq 1} H_\Omega(y_n, p_n) <\infty.
\end{align*}
Thus 
\begin{align*}
\sup_{n \geq 1} H_\Omega(g_n^{-1} (y_n), p_0) <\infty.
\end{align*}
So there exists $n_j \rightarrow\infty$ so that the limit 
\begin{align*}
q:=\lim_{j \rightarrow \infty} g_{n_j}^{-1}(y_{n_j})
\end{align*}
exists in $\Omega$. Notice that $q \notin \Pb(\ker T)$ by Proposition~\ref{prop:dynamics_of_automorphisms} and so the ``moreover'' part of Observation~\ref{obs:limit of endomorphisms} implies that 
\begin{align*}
T(q) = \lim_{n \rightarrow \infty} g_n(q) = \lim_{j \rightarrow \infty} g_{n_j}(q)=\lim_{j \rightarrow \infty} g_{n_j}\left( g_{n_j}^{-1}(y_{n_j}) \right)= \lim_{j \rightarrow \infty}y_{n_j} = y.
\end{align*}

Since $y$ was arbitrary, $F_\Omega(x) \subset T(\Omega)$.

\end{proof}

\subsection{Proximal elements}\label{sec:prox_elements_basics}

In this section we recall some basic properties of proximal elements. For more background we refer the reader to~\cite{BQ2016}.

\begin{definition} Suppose that $F : M \rightarrow M$ is a $C^1$ map of a manifold $M$. Then a fixed point $x \in M$ of $F$ is \emph{attractive} if $\abs{\lambda} < 1$ for every eigenvalue $\lambda$ of $d(F)_x : T_x M \rightarrow T_x M$.
\end{definition}

A straightforward calculation provides the following characterization of proximality. 

\begin{observation}\label{obs:att_fixed_pt} Suppose that $g \in \PGL_d(\Rb)$ and $x$ is a fixed point of the $g$ action on $\Pb(\Rb^d)$. Then the following are equivalent: 
\begin{enumerate}
\item $x$ is an attractive fixed point of $g$, 
\item $g$ is proximal and $x = \ell_g^+$. 
\end{enumerate}
\end{observation}

Next we explain the global dynamics of a proximal element. 

\begin{definition}
If $g \in \PGL_d(\Rb)$ is proximal, then define $H_g^- \in \Gr_{d-1}(\Rb^d)$ to be the unique $g$-invariant linear hyperplane with 
\begin{align*}
\ell_g^+ \oplus H_g^- = \Rb^d.
\end{align*} 
If $g$ is bi-proximal, then also define $H_g^+ := H^-_{g^{-1}}$. 
\end{definition}

When $g \in \PGL_d(\Rb)$ is proximal, $H_g^-$ is usually called the repelling hyperplane of $g$. This is motivated by the following observation.

\begin{observation}\label{obs:iterates_of_proximal} If $g \in \PGL_d(\Rb)$ is proximal, then 
\begin{align*}
T_g := \lim_{n \rightarrow \infty} g^n
\end{align*}
exists in $\Pb(\End(\Rb^d))$. Moreover, $\image T_g =\ell_g^+$, $\ker T_g = H_g^-$, and
\begin{align*}
\image T_g  \oplus \ker T_g = \Rb^d.
\end{align*}
Hence 
\begin{align*}
\ell_g^+ = \lim_{n \rightarrow \infty} g^nx
\end{align*}
for all $x \in \Pb(\Rb^d) \setminus \Pb(H^-_g)$. 
\end{observation}

We observe the following. 

\begin{observation}\label{obs:fixed pts of proximal elements are extreme} Suppose that $\Omega \subset \Pb(\Rb^d)$ is a properly convex domain. If $g \in \Aut(\Omega)$ is proximal, then $\ell_g^+$ is an extreme point of $\partial \Omega$ and $ \Pb(H^-_g) \cap \partial \Omega = \emptyset$. 
\end{observation} 

\begin{proof} Proposition~\ref{prop:dynamics_of_automorphisms} implies that $\ell_g^+ \in \partial \Omega$ and $ \Pb(H^-_g) \cap \partial \Omega = \emptyset$. Let $F=F_\Omega(\ell^+_g)$ and $V = \Spanset F$. Then $g(V) = V$. Let $\overline{g} \in \GL_d(\Rb)$ be a lift of $g \in \PGL_d(\Rb)$ and let $h \in \GL(V)$ denote the element obtained by restricting $\overline{g}$ to $V$. Notice that $h$ is proximal since $\ell^+_g \subset V$. Further $[h] \in \Aut(F)$ and $h(\ell_g^+) = \ell_g^+$. Since $\Aut(F)$ acts properly on $F$ and $\ell_g^+ \in F$, the cyclic group 
\begin{align*}
[h]^{\Zb} \leq \Aut(F) \leq \PGL(V)
\end{align*}
must be relatively compact. This implies that every eigenvalue of $h$ has the same absolute value. Then, since $h$ is proximal, $V$ must be one-dimensional and so $F=\{\ell_g^+\}$. Thus $\ell_g^+$ is an extreme point. 

\end{proof} 

The following result can be viewed as a converse to Observation~\ref{obs:iterates_of_proximal}  and  will be used to construct proximal elements.

\begin{proposition}\label{prop:crit_for_prox} Suppose that $(g_n)_{n \geq 1}$ is a sequence in $\PGL_d(\Rb)$ and 
\begin{align*}
T := \lim_{n \rightarrow \infty} g_n
\end{align*}
exists in $\Pb(\End(\Rb^d))$. If $\dim(\image T )= 1$ and 
\begin{align*}
\image T  \oplus \ker T = \Rb^d,
\end{align*}
then for $n$ sufficiently large $g_n$ is proximal and
\begin{align*}
\image T  = \lim_{n \rightarrow \infty} \ell_{g_n}^+.
\end{align*}
\end{proposition}

\begin{proof} Since $g_n \rightarrow T$ in $\Pb(\End(\Rb^d))$,
\begin{align*}
\lim_{n \rightarrow \infty} g_n(x) = T(x) = \image T \in \Pb(\Rb^d)
\end{align*}
for all $x \in \Pb(\Rb^d) \setminus \Pb(\ker T)$. Moreover, the convergence is uniform on compact subsets of $\Pb(\Rb^d) \setminus \Pb(\ker T)$.

By assumption
\begin{align*}
\image T  \notin \Pb(\ker T),
\end{align*}
so we can find a compact neighborhood $U$ of $\image T$ in $\Pb(\Rb^d)$ such that $U$ is homeomorphic to a closed ball and 
\begin{align*}
U  \cap \Pb(\ker T) = \emptyset.
\end{align*}
Then by passing to a tail, we can assume that $g_n(U) \subset U$ for all $n$. So by the Brouwer fixed-point theorem, each $g_n$ has a fixed point $x_n \in U$. Since $U$ can be chosen arbitrarily small we also have 
\begin{align*}
\image T = \lim_{n \rightarrow \infty} x_n.
\end{align*}

We claim that for $n$ large, $x_n$ is an attractive fixed point of $g_n$. By Lemma~\ref{obs:att_fixed_pt} this will finish the proof. Let $f_n : \Pb(\Rb^d) \rightarrow \Pb(\Rb^d)$ be the diffeomorphism induced by $g_n$, that is $f_n(x) = g_n(x)$ for all $x$. Then, since each $g_n$ acts by projective linear transformations, we see that $f_n$ converges locally uniformly in the $C^\infty$ topology on $\Pb(\Rb^d) \setminus \Pb(\ker T)$ to the constant map $f \equiv \image T$. So fixing a Riemannian metric on $\Pb(\Rb^d)$ we have 
\begin{align*}
\lim_{n \rightarrow \infty} \norm{d(f_n)_{x_n}} = 0.
\end{align*}
Hence for $n$ large $x_n$ is an attractive fixed point of $g_n$.
\end{proof}

\subsection{Rank one isometries}

In this section we state a characterization of rank one isometries established in~\cite{I2019}. 

\begin{theorem}[{M. Islam~\cite[Proposition 6.3]{I2019}}]\label{thm:char_of_rank_one}
Suppose that $\Omega \subset \Pb(\Rb^d)$ is a properly convex domain and $\gamma \in \Aut(\Omega)$. If 
\begin{align*}
\inf_{p \in \Omega} H_\Omega(\gamma (p),p) > 0
\end{align*}
and $\gamma$ fixes two points $x,y \in \partial \Omega$ with $s_{\partial \Omega}(x, y) > 2$, then 
\begin{enumerate}
\item $\gamma$ is bi-proximal and $\{\ell_\gamma^+, \ell_\gamma^-\} = \{x,y\}$. In particular, $\gamma$ is a rank one isometry. 
\item The only points fixed by $\gamma$ in $\partial \Omega$ are $\ell_\gamma^+$ and $\ell_\gamma^-$. 
\item If $w \in \partial \Omega$, then
\begin{align*}
(\ell_\gamma^+,w) \cup (w,\ell_\gamma^-) \subset \Omega.
\end{align*}
\item If $z \in \partial \Omega \setminus \{\ell^\pm_\gamma\}$, then
\begin{align*}
s_{\partial \Omega}(\ell_\gamma^\pm, z) =\infty.
\end{align*}
\end{enumerate}
\end{theorem}

\begin{remark} Notice that (4) is a consequence of (3). \end{remark}

\section{A semigroup associated to a boundary face}\label{sec:semigroups_associated_to_face}

In this section we prove Theorem~\ref{thm:structure_subsets_intro}, which we restate here. 

\begin{theorem}\label{thm:structure_subsets} Suppose that $\Omega \subset \Pb(\Rb^d)$ is an irreducible properly convex domain and $\Gamma \leq \Aut(\Omega)$ is a discrete group acting co-compactly on $\Omega$. If $\Omega$ is non-symmetric, $F \subset \partial\Omega$ is a boundary face, $V:=\Spanset F$, and $\dim(V) \geq 2$, then: 
\begin{enumerate}[(a)]
\item If $T \in \overline{\Gamma}^{\End}_F$, then $T(\Omega) \subset \overline{F}$.
\item If $T \in \overline{\Gamma}^{\End}_{F,\star}$, then $T(F)$ is an open subset of $F$. 
\item The set 
\begin{align*}
\left\{ T|_V : T \in \overline{\Gamma}^{\End}_{F,\star}\right\}
\end{align*}
is a non-discrete Zariski dense semigroup in $\Pb(\End(V))$.
\end{enumerate}
\end{theorem}

The proof of Theorem~\ref{thm:structure_subsets} will follow from a series of lemmas, many of which hold in greater generality. 

For the rest of the section fix a properly convex domain $\Omega \subset \Pb(\Rb^d)$ and a subgroup $\Gamma \leq \Aut(\Omega)$. Notice that we are not (currently) assuming that $\Omega$ is irreducible, that $\Gamma$ is discrete, or that $\Gamma$ acts co-compactly on $\Omega$. 

We begin with an observation. 

\begin{observation}\label{obs:structure_subsets} \
\begin{enumerate}[(a)]
\item If $T \in \overline{\Gamma}^{\End}$, then $\Pb(\ker T) \cap \Omega = \emptyset$. 
\item If $S,T \in \overline{\Gamma}^{\End}$ and $\image T \setminus \ker S \neq \emptyset$, then $S \circ T \in \overline{\Gamma}^{\End}$. 
\end{enumerate}
\end{observation}

\begin{proof}
Part (a) follows immediately from Proposition~\ref{prop:dynamics_of_automorphisms}.

For part (b), fix $S,T \in \overline{\Gamma}^{\End}$ with $\image T \setminus \ker S \neq \emptyset$. By hypothesis $S \circ T$ is a well defined element of $\Pb(\End(\Rb^d))$. To show that $S \circ T \in \overline{\Gamma}^{\End}$, fix sequences $(g_n)_{n \geq 1}$, $(h_n)_{n \geq 1}$ in $\Gamma$ such that 
\begin{align*}
S = \lim_{n \rightarrow \infty} g_n \quad \text{and} \quad T = \lim_{n \rightarrow \infty} h_n
\end{align*}
in $\Pb(\End(\Rb^d))$. Then, since $S \circ T \neq 0$, we have 
\begin{align*}
S \circ T = \lim_{n \rightarrow \infty} g_nh_n
\end{align*}
in $\Pb(\End(\Rb^d))$. So $S \circ T \in \overline{\Gamma}^{\End}$.

\end{proof}

\begin{lemma}\label{lem:semi_gp_basic_prop_1} If $F \subset \partial\Omega$ a boundary face and $T \in \overline{\Gamma}^{\End}_F$, then $T(\Omega) \subset \overline{F}$.
\end{lemma}

\begin{proof}  Suppose $T \in \overline{\Gamma}^{\End}_F$. Then there exists a sequence $(g_n)_{n \geq 1}$ in  $\Gamma$ such that 
\begin{align*}
T = \lim_{n \rightarrow \infty} g_n 
\end{align*}
in $\Pb(\End(\Rb^d))$. Since $\Pb(\ker T) \cap \Omega = \emptyset$ we have
\begin{align*}
T(p) = \lim_{n \rightarrow \infty} g_n(p) \in \overline{\Omega}
\end{align*}
for all $p \in \Omega$. So $T(\Omega) \subset \overline{\Omega}$. Since ${\rm image}(T) \subset V$ we then have
\begin{equation*}
T(\Omega) \subset \Pb(V) \cap \overline{\Omega} = \overline{F}. \qedhere
\end{equation*}
\end{proof}

\begin{lemma}\label{lem:semi_gp_basic_prop_2} If $F \subset \partial\Omega$ a boundary face and $T \in \overline{\Gamma}^{\End}_{F,\star}$, then $T(F)$ is an open subset of $F$. 
\end{lemma}

\begin{proof} By definition and Observation~\ref{obs:structure_subsets}
\begin{align*}
(\Omega \cup F) \cap \Pb(\ker T) \subset (\Omega \cup \Pb(V)) \cap \Pb(\ker T) = \emptyset.
\end{align*}
So $T$ induces a continuous map on $ \Omega \cup F$. Since $F \subset \overline{\Omega}$, the previous lemma implies that
\begin{align*}
T(F) \subset \overline{T(\Omega)} \subset \overline{F}.
\end{align*}
Since $V \cap \ker T = \{0\}$, $T(F)$ is an open subset of $\Pb(V)$. So 
\begin{equation*}
T(F) \subset \relint\left(\overline{F}\right) = F. \qedhere
\end{equation*}
\end{proof}

\begin{lemma} If $F \subset \partial\Omega$ a boundary face, then the set 
\begin{align*}
\left\{ T|_V : T \in \overline{\Gamma}^{\End}_{F,\star}\right\}
\end{align*}
is a semigroup in $\Pb(\End(V))$.
\end{lemma}

\begin{proof}Fix $T_1, T_2 \in \overline{\Gamma}^{\End}_{F,\star}$. Then
\begin{align*}
\image T_2 \setminus \ker T_1 = V \setminus \ker T_1 = V \setminus \{0\} \neq \emptyset
\end{align*}
and so $T_1 \circ T_2 \in \overline{\Gamma}^{\End}$ by Observation~\ref{obs:structure_subsets}. 

We first show that $\ker (T_1 \circ T_2) \cap V =\{0\}$. Suppose $v \in \ker (T_1 \circ T_2) \cap V$. Then $T_2(v) \in \ker T_1$. But $\image T_2 = V$ and $\ker T_1 \cap V =\{0\}$, so $T_2(v) = 0$. So $v \in \ker T_2 \cap V=\{0\}$. So $v=0$. Thus 
\begin{align}
\label{eq:ker_transverse}
\{0\} =  \ker (T_1 \circ T_2) \cap V.
\end{align}
Next, by definition,
\begin{align*}
\image(T_1 \circ T_2) \subset \image T_1 = V.
\end{align*}
 So by Equation~\eqref{eq:ker_transverse} and dimension counting we have
\begin{align*}
\image(T_1 \circ T_2) = V.
\end{align*}
Thus $T_1 \circ T_2 \in  \overline{\Gamma}^{\End}_{F,\star}$. 

Since $\image T_2 = V$ we also have 
\begin{align*}
T_1|_V \circ T_2|_V = (T_1 \circ T_2)|_V.
\end{align*}
So
\begin{align*}
(T_1 \circ T_2)|_V \in \left\{ T|_V : T \in \overline{\Gamma}^{\End}_{F,\star}\right\}.
\end{align*}

Then, since $T_1, T_2 \in \overline{\Gamma}^{\End}_{F,\star}$ were arbitrary, we see that 
\begin{align*}
\left\{ T|_V : T \in \overline{\Gamma}^{\End}_{F,\star}\right\}
\end{align*}
is a semigroup in $\Pb(\End(V))$.
\end{proof}

The next lemma requires a definition. 

\begin{definition} A point $x \in \partial \Omega$ is a \emph{conical limit point of $\Gamma$} if there exist $p_0 \in \Omega$, a sequence $(p_n)_{n \geq 1}$ in $[p_0,x)$ with $p_n \rightarrow x$, and a sequence $(\gamma_n)_{n \geq 1}$ in $\Gamma$ with 
\begin{align*}
\sup_{n \geq 1} H_\Omega(\gamma_{n}(p_0), p_n) <+\infty.
\end{align*}
\end{definition} 

Notice that if $\Gamma$ acts co-compactly on $\Omega$, then every boundary point is a conical limit point. 

\begin{lemma}\label{lem:techincal_semigroup_existence} Suppose that $x \in \partial \Omega$ is a conical limit point of $\Gamma$, $F = F_\Omega(x)$, $V = \Spanset F$, and $\dim(V) = k$. If $k \geq 2$ and the image of $\Gamma \hookrightarrow \PGL(\wedge^k \Rb^d)$ is strongly irreducible (e.g. $\Gamma$ is Zariski dense in $\PGL_d(\Rb)$), then there exists a sequence $(g_n)_{n \geq 1}$ in   $\Gamma$ with the following properties:
\begin{enumerate}
\item $ g_n \rightarrow T$  in $\Pb(\End(\Rb^d))$ where $T \in \overline{\Gamma}^{\End}_{F,\star}$.  
\item $g_1|_V, g_2|_V, \dots$ are pairwise distinct elements of $\Pb({\rm Lin}(V,\Rb^d))$. 
\end{enumerate}
\end{lemma}

\begin{proof} By hypothesis, there exist $p_0 \in \Omega$,  a sequence $(p_n)_{n \geq 1}$ in  $[p_0,x)$ with $p_n \rightarrow x$, and a sequence $(\gamma_n)_{n \geq 1}$ in $\Gamma$ with 
\begin{align*}
\sup_{n \geq 1} H_\Omega(\gamma_{n}(p_0), p_n) <+\infty.
\end{align*}
After passing to a subsequence we can suppose that the limit
\begin{align*}
S = \lim_{n \rightarrow \infty} \gamma_{n}
\end{align*}
exists in $\Pb(\End(\Rb^d))$. Then, by Proposition~\ref{prop:dynamics_of_automorphisms_2},
\begin{align*}
\image S = \Spanset F = V
\end{align*}
and so $S \in \overline{\Gamma}^{\End}_F$. By passing to another subsequence we can suppose that 
 \begin{align*}
 V_\infty = \lim_{n \rightarrow \infty} \gamma_n^{-1} V
 \end{align*}
 exists in $\Gr_{k}(\Rb^d)$. 
 
 Suppose $V= \Spanset\{ v_1,\dots, v_k\}$, $V_\infty = \Spanset\{ u_1,\dots, u_k\}$, and $\ker S = \Spanset\{ s_1,\dots, s_{d-k}\}$. Let $W_1 = [u_1 \wedge \dots \wedge u_k]$ and 
 $$
 W_2 = \left\{ \alpha \in \wedge^k \Rb^d : \alpha \wedge s_1 \wedge \dots \wedge s_{d-k} = 0\right\}.
 $$
 Since the image of $\Gamma \hookrightarrow \PGL(\wedge^k \Rb^d)$ is strongly irreducible, Observation~\ref{obs:str_irred_transverse} implies that there exists $\phi \in \Gamma$ such that $\phi[ v_1 \wedge \dots \wedge v_k] \notin W_1 \cup W_2$. Equivalently, $\ker S \cap \phi V= \{0\}$ and $\phi V \neq V_\infty$. 
 
Then define $g_n:=\gamma_n \phi$. Then 
 \begin{align*}
T:=S \circ \phi= \lim_{n \rightarrow \infty}  g_n
\end{align*}
 exists in $\Pb(\End(\Rb^d))$. Further, $\image T = \image S=V$ and 
 \begin{align*}
 \ker T \cap V = \phi^{-1} \left( \ker S \cap \phi V \right) = \{ 0\}.
 \end{align*}
 So $T \in \overline{\Gamma}^{\End}_{F,\star}$. Also, since $T(V) = V$, we have
 \begin{align*}
 V = T(V) = \lim_{n \rightarrow \infty} g_n V.
 \end{align*}
 
 Next we claim that $g_n V \neq V$ for $n$ sufficiently large. Notice that $g_n V = V$ if and only if $g_n^{-1}V = V$ if and only if $\gamma_n^{-1} V = \phi V$. But $\gamma_n^{-1} V \rightarrow V_\infty$ and $\phi V \neq V_\infty$. So $g_n V \neq V$ for $n$ sufficiently large. 
 
Finally, since $g_n V \rightarrow V$ and $g_n V \neq V$  for $n$ sufficiently large, we can pass to a subsequence so that $V, g_1V, g_2V, \dots$ are pairwise distinct subspaces. Thus $g_1|_V, g_2|_V, \dots$ must be pairwise distinct.
\end{proof}

\begin{lemma}\label{lem:face_semi_gp_non_discrete} Suppose that $x \in \partial \Omega$ is a conical limit point of $\Gamma$, $F = F_\Omega(x)$, $V = \Spanset F$, and $\dim(V) = k$. If $k \geq 2$ and the image of $\Gamma \hookrightarrow \PGL(\wedge^k \Rb^d)$ is strongly irreducible (e.g. $\Gamma$ is Zariski dense in $\PGL_d(\Rb)$), then the set
\begin{align*}
\left\{ T|_V : T \in \overline{\Gamma}^{\End}_{F,\star}\right\}
\end{align*}
is non-discrete in $\Pb(\End(V))$.
\end{lemma}

\begin{proof}  Let $T \in \overline{\Gamma}^{\End}_{F,\star}$ and $(g_n)_{n \geq 1}$ be as in the previous lemma. Since $g_1|_V, g_2|_V, \dots$ are pairwise distinct and each $g_n|_V$ is determined by its values on any set of $\dim V+1$ points in general position, after passing to a subsequence we can find a point $x_0 \in F$ such that $g_1(x_0), g_2(x_0), \dots$ are pairwise distinct. 

Since $x_0 \in F$ and $\Pb(\ker T) \cap F = \emptyset$, we have
\begin{align*}
T(x_0) = \lim_{n \rightarrow \infty} g_n(x_0).
\end{align*}
Since $g_1(x_0), g_2(x_0), \dots$ are pairwise distinct, by passing to another sequence we can assume that $g_n(x_0) \neq T(x_0)$ for all $n$. Then, for each $n$ there exists a unique projective line $L_n$ containing $T(x_0)$ and $g_n(x_0)$. By passing to a subsequence we can suppose that $L_n$ converges to a projective line $L$. Then let $W \subset \Rb^d$ be the two dimensional linear subspace with $L=\Pb(W)$.

Fix some $W^\prime \in \Gr_k(\Rb^d)$ with $W \subset W^\prime$. Suppose $V= \Spanset\{ v_1,\dots, v_k\}$, $W^\prime = \Spanset\{ w_1,\dots, w_k\}$, and $\ker T = \Spanset\{ t_1,\dots, t_{d-k}\}$. Let 
 $$
 U = \left\{ \alpha \in \wedge^k \Rb^d : \alpha \wedge t_1 \wedge \dots \wedge t_{d-k} = 0\right\}.
 $$
 Since the image of $\Gamma \hookrightarrow \PGL(\wedge^k \Rb^d)$ is strongly irreducible, Observation~\ref{obs:str_irred_transverse} implies that there exists $\varphi \in \Gamma$ such that $\varphi[ v_1 \wedge \dots \wedge v_k] \notin U$ and $\varphi[w_1 \wedge \dots \wedge w_k] \notin U$. Hence $\ker T \cap \varphi V = \{0\}$ and $\ker T \cap \varphi W = \{0\}$.

Notice that $T \varphi T = \lim_{n \rightarrow \infty} g_n \varphi g_n$ is in $\overline{\Gamma}^{\End}_{F,\star}$. Then replacing $(g_n)_{n \geq 1}$ with a tail, we can assume that 
\begin{align*}
S_n : = T \varphi g_n \in \overline{\Gamma}^{\End}_{F, \star}
\end{align*}
for all $n$.

We claim that the set 
\begin{align*}
\{ S_n(x_0) : n \geq 0\} \subset F
\end{align*}
is infinite. For this calculation we fix an affine chart $\Ab$ of $\Pb(\Rb^d)$ which contains $\overline{\Omega}$. We then identify $\Ab$ with $\Rb^{d-1}$ such that  $T(x_0)=0$ and 
\begin{align*}
\Ab \cap L= \{ (t, 0,\dots,0) : t \in \Rb\}.
\end{align*}
Since $\ker T \cap \varphi V = \{0\}$, in these coordinates the map $T\varphi$ is smooth in a neighborhood of $0=T(x_0)$. Further, since $\ker T \cap \varphi W = \{0\}$, in these coordinates 
\begin{align*}
d(T\varphi)_0(1,0,\dots,0) \neq 0.
\end{align*}
Now, since $L_n \rightarrow L$ and $g_n(x_0) \rightarrow T(x_0)$, in these coordinates 
\begin{align*}
g_n(x_0) = (t_n,0,\dots,0) + {\rm o}(\abs{t_n}).
\end{align*}
for some sequence $(t_n)_{n \geq 1}$ converging to 0. Then in these coordinates 
\begin{align*}
S_n(x_0) & = T\varphi g_n(x_0) = T\varphi\Big( ( t_n,0,\dots,0)+ {\rm o}(\abs{t_n}) \Big) \\
& = T\varphi T(x_0) + t_nd(T\varphi)_0(1,0,\dots,0) +  {\rm o}(\abs{t_n}).
\end{align*}
Since $d(T\varphi)_0(1,0,\dots,0) \neq 0$ and $t_n \rightarrow 0$, we see that the set $\{ S_n(x_0) : n \geq 0\}$ is infinite.

Finally, since $S_n|_V \rightarrow T \varphi T|_V$, this implies that 
\begin{align*}
\left\{ S_n|_V : n \geq 0 \right\} \cup \left\{ T \varphi T|_V \right\}
\end{align*}
is non-discrete in $\Pb(\End(V))$. 
\end{proof}

\begin{lemma}\label{lem:semi_gp_basic_prop_3} Suppose that $x \in \partial \Omega$ is a conical limit point of $\Gamma$, $F = F_\Omega(x)$, $V = \Spanset F$, and $\dim(V) = k$. If $k \geq 2$ and $\Gamma$ is Zariski dense in $\PGL_d(\Rb)$, then
\begin{align*}
\left\{ T|_V : T \in \overline{\Gamma}^{\End}_{F,\star}\right\}
\end{align*}
is Zariski dense in $\Pb(\End(V))$.
\end{lemma}

\begin{proof}
Let $Z_0$ be the Zariski closure of 
\begin{align*}
\left\{ T|_V : T \in \overline{\Gamma}^{\End}_{F,\star}\right\}
\end{align*}
in $\Pb(\End(V))$. 

Lemma~\ref{lem:techincal_semigroup_existence} implies that $\overline{\Gamma}^{\End}_{F,\star}$ is non-empty. So fix $T \in \overline{\Gamma}^{\End}_{F,\star}$. Then define
\begin{align*}
Z_1 = \{ g \in \PGL_d(\Rb) : {\rm rank} (T \circ g|_V) < \dim(V) \}.
\end{align*}
Notice that $Z_1$ is a proper Zariski closed set in $\PGL_d(\Rb)$ since ${\rm rank}(T) = \dim(V)$. Also define
\begin{align*}
Z_2 = \{ g \in \PGL_d(\Rb) : T \circ g|_V \in Z_0\}.
\end{align*}
Notice that $Z_2$ is a Zariski closed subset of $\PGL_d(\Rb)$. 

We claim that $\Gamma \subset Z_1 \cup Z_2$. If $g \in \Gamma \setminus Z_1$, then ${\rm rank}(T \circ g|_V) = \dim V$ and 
$$
\image(T \circ g|_V) \subset \image T = V.
$$
So $(T \circ g)(V) = V$ which implies that $T \circ g \in   \overline{\Gamma}^{\End}_{F,\star}$ and hence that $g \in Z_2$. So $\Gamma \subset Z_1 \cup Z_2$.

Then, since $Z_1$ is a proper Zariski closed subset of $\PGL_d(\Rb)$ and $\Gamma$ is Zariski dense in $\PGL_d(\Rb)$, we see that $Z_2 = \PGL_d(\Rb)$. Then 
$$
Z_0 \supset \{ T \circ g|_V  : g \in Z_2\} =  \{ T \circ g|_V  : g \in \PGL_d(\Rb) \} \supset \PGL(V)
$$
since $\image T = V$. Thus  $Z_0 = \Pb(\End(V))$. 
\end{proof}

\begin{proof}[Proof of Theorem~\ref{thm:structure_subsets}] Parts (1) and (2) follow from Lemmas~\ref{lem:semi_gp_basic_prop_1} and~\ref{lem:semi_gp_basic_prop_2} respectively. Since $\Gamma$ acts co-compactly on $\Omega$ every point in $\partial \Omega$ is a conical limit point and Theorem~\ref{thm:Zariski_closure}  implies that $\Gamma$ is Zariski dense in $\PGL_d(\Rb)$. So Part (3) follows from Lemmas~\ref{lem:semi_gp_basic_prop_1}, ~\ref{lem:face_semi_gp_non_discrete}, and~\ref{lem:semi_gp_basic_prop_3}.

\end{proof} 

\section{The main rigidity theorem}\label{sec:main_rigidity_result}

Recall that $\Ec_\Omega \subset \partial \Omega$ denotes the set of extreme points of a properly convex domain $\Omega$. In this section we prove the following rigidity result. 

\begin{theorem}\label{thm:main_rigidity_result} Suppose that $\Omega \subset \Pb(\Rb^d)$ is an irreducible properly convex divisible domain and there exists a boundary face $F \subset \partial \Omega$ such that 
\begin{align*}
F \cap \overline{\Ec_\Omega} = \emptyset.
\end{align*}
Then $\Omega$ is symmetric with real rank at least two.
\end{theorem}

The rest of the section is devoted to the proof of the theorem. So suppose $\Omega \subset \Pb(\Rb^d)$ satisfies the hypothesis of the Theorem. Then let $\Gamma \leq \Aut(\Omega)$ be a discrete group acting co-compactly on $\Omega$. 

We assume for a contradiction that $\Omega$ is not symmetric with real rank at least two.

\begin{lemma} $\Omega$ is not symmetric. \end{lemma}

\begin{proof} If $\Omega$ were symmetric, then by assumption it would have real rank one. Then, by the characterization of symmetric convex divisible domains, $\Omega$ coincides with the unit ball in some affine chart. Then $\Ec_\Omega = \partial \Omega$ which is impossible since there exists a boundary face $F \subset \partial \Omega$ such that
\begin{equation*}
F \cap \overline{\Ec_\Omega} = \emptyset. \qedhere
\end{equation*}
  \end{proof}

Now we fix a boundary face $F \subset \partial\Omega$ where 
\begin{align*}
\overline{\Ec_{\Omega}} \cap F = \emptyset
\end{align*}
and if $F^\prime \subset \partial \Omega$ is a face with $\dim F^\prime < \dim F$, then 
\begin{align*}
\overline{\Ec_{\Omega}} \cap F^\prime \neq \emptyset.
\end{align*}
Then define $V := \Spanset F$.

\begin{lemma}\label{lem:auto} If $T \in \overline{\Gamma}^{\End}_{F,\star}$, then the map 
\begin{align*}
p \in F \rightarrow T(p) \in \Pb(V)
\end{align*}
is in $\Aut(F)$. \end{lemma}

\begin{proof} Notice that $T|_V \in \PGL(V)$ since $T(V) \subset V$ and $\ker T \cap V = \{0\}$. So we just have to show that $T(F)=F$.  Theorem~\ref{thm:structure_subsets} part (2) says that $T(F) \subset F$ and so we just have to show that $F \subset T(F)$. 

Fix $y \in F$. Since the set $T(\overline{F}) \cap F$ is closed in $F$, there exists $x_0 \in T(\overline{F}) \cap F$ such that 
\begin{align*}
H_F(y,x_0) = \min_{x \in T(\overline{F}) \cap F} H_F(y,x).
\end{align*}
Since $T|_V \in \PGL(V)$, the set $T(F)$ is open in $F$. So we either have $y=x_0 \in T(F)$ or $x_0 \in T(\partial F)$. Suppose for a contradiction that $x_0 \in T(\partial F)$. Then let $x_0^\prime \in \partial F$ be the point where $T(x_0^\prime)=x_0$. Next let $F^\prime \subset \partial F$ be the face of $x_0^\prime$. Then $\dim F^\prime < \dim F$, so 
\begin{align*}
\overline{\Ec_{\Omega}} \cap F^\prime \neq \emptyset.
\end{align*}
Thus we can find $z \in F^\prime$ and a sequence $(z_n)_{n \geq 1}$ in  $\Ec_{\Omega}$ such that $z_n \rightarrow z$. Since $z \in F^\prime$, there exists an open line segment $L$ in $\overline{F}$ which contains $z$ and $x_0^\prime$. Then $T(L)$ is an open line segment in $\overline{F}$ since $T|_V \in \PGL(V)$. Then, since $T(x_0^\prime) \in F$, we also have $T(z) \in F$. Since 
\begin{align*}
T \in \overline{\Gamma}^{\End}_{F,\star} \subset \overline{\Gamma}^{\End}
\end{align*} 
there exists a sequence $(g_n)_{n \geq 1}$ in $\Gamma$ such that $g_n \rightarrow T$ in $\Pb(\End(\Rb^d))$. Now $z \notin \Pb(\ker T)$ since $\ker T \cap V =\{0\}$. So by the ``moreover'' part of Observation~\ref{obs:limit of endomorphisms}
\begin{align*}
T(z) =  \lim_{n \rightarrow \infty} g_n(z_n) \in F.
\end{align*}
However, $g_n(z_n) \in \Ec_{\Omega}$ and so
\begin{align*}
T(z) \in \overline{\Ec_{\Omega}} \cap F = \emptyset.
\end{align*}
Thus we have a contradiction. Hence $y=x_0 \in T(F)$. Since $y \in F$ was arbitrary we have $F \subset T(F)$. 
\end{proof}

\begin{lemma} $\Aut(F)$ is  non-discrete and Zariski dense in $\PGL(V)$. \end{lemma}

\begin{proof} This follows immediately from Lemma~\ref{lem:auto} and Theorem~\ref{thm:structure_subsets} part (3). 
\end{proof}

\begin{lemma}\label{lem:big_Z_closure} $\PSL(V) \subset \Aut(F)$. \end{lemma}

\begin{proof} Let $\Aut_0(F)$ denote the connected component of the identity in $\Aut(F)$ and let $\gL \subset \mathfrak{sl}(V)$ denote the Lie algebra of $\Aut_0(F)$. Then $\gL \neq \{0\}$ since $\Aut(F)$ is closed and non-discrete. Also $\Aut_0(F)$ is normalized by $\Aut(F)$ and so
\begin{align*}
{\rm Ad}(g) \gL = \gL
\end{align*}
for all $g \in \Aut(F)$. Then, since $\Aut(F)$ is Zariski dense in $\PGL(V)$, we see that
\begin{align*}
{\rm Ad}(g) \gL = \gL
\end{align*}
for all $g \in \PGL(V)$. Since the representation ${\rm Ad} : \PGL(V) \rightarrow \GL(\mathfrak{sl}(V))$ is irreducible, we must have $\gL = \mathfrak{sl}(V)$. Thus $\Aut_0(F) = \PSL(V)$. 

\end{proof}

\begin{proof}[Proof of Theorem~\ref{thm:main_rigidity_result}] 
The previous lemma immediately implies a contradiction: fix $x \in F$, then 
\begin{align*}
\Pb(V) \supset F \supset \Aut(F) \cdot x \supset \PSL(V) \cdot x = \Pb(V).
\end{align*}
So $F=\Pb(V)$ which contradicts the fact that $\Omega$ is properly convex. 
\end{proof}

\section{Density of bi-proximal elements}\label{sec:density_of_fixed_pts}

In this section we prove a density result for the attracting and repelling fixed points of bi-proximal elements. To state the result we need one definition. If $\Omega \subset \Pb(\Rb^d)$ is a properly convex domain and $\Gamma \leq \Aut(\Omega)$, then the \emph{limit set of $\Gamma$} is 
$$
\Lc_\Omega(\Gamma) = \bigcup_{p \in \Omega} \overline{\Gamma \cdot p} \cap \partial \Omega.
$$
Equivalently, a point $x \in \partial \Omega$ is in $\Lc_\Omega(\Gamma)$ if and only if there exist $p \in \Omega$ and a sequence $(\gamma_n)_{n \geq 1}$ in $\Gamma$ such that $\gamma_n(p) \rightarrow x$. 

\begin{theorem}\label{thm:biprox_density} Suppose that $\Omega \subset \Pb(\Rb^d)$ is a properly convex domain and $\Gamma \leq \Aut(\Omega)$ is a strongly irreducible group. If $x,y \in \Lc_\Omega(\Gamma)$ are extreme points of $\Omega$ and $(x,y) \subset \Omega$, then there exists a sequence of bi-proximal elements $(g_n)_{n \geq 1}$ in  $\Gamma$ such that 
\begin{align*}
\lim_{n \rightarrow \infty} \ell^+_{g_n} = x \quad \text{and} \quad \lim_{n \rightarrow \infty} \ell^-_{g_n} = y. 
\end{align*}
\end{theorem}

Before proving the theorem we state and prove one corollary. 

\begin{corollary}\label{cor:biprox_density} Suppose that $\Omega \subset \Pb(\Rb^d)$ is an irreducible properly convex domain and $\Gamma \leq \Aut(\Omega)$ is a discrete group that acts co-compactly on $\Omega$. If $x,y \in \partial \Omega$ are extreme points and $(x,y) \subset \Omega$, then there exists a sequence of bi-proximal elements $(g_n)_{n \geq 1}$ in  $\Gamma$
\begin{align*}
\lim_{n \rightarrow \infty} \ell^+_{g_n} = x \quad \text{and} \quad \lim_{n \rightarrow \infty} \ell^-_{g_n} = y. 
\end{align*}
\end{corollary}

\begin{proof}[Proof of Corollary~\ref{cor:biprox_density}] A result of Vey~\cite[Theorem 5]{V1970} implies that $\Gamma$ is strongly irreducible. Proposition~\ref{prop:dynamics_of_automorphisms_2} implies that $\partial \Omega = \Lc_\Omega(\Gamma)$. So Theorem~\ref{thm:biprox_density}  implies the corollary. 
\end{proof}

\begin{proof}[Proof of Theorem~\ref{thm:biprox_density}] By definition there exist $p \in\Omega$ and a sequence $(\gamma_n)_{n \geq 1}$ in  $\Gamma$ such that $\gamma_n(p) \rightarrow x$. Passing to a subsequence we can suppose the limits 
$$
T^+ = \lim_{n \rightarrow \infty} \gamma_n \quad \text{and}  \quad T^- = \lim_{n \rightarrow \infty} \gamma_n^{-1}
$$
exist in $\Pb(\End(\Rb^d))$. By Proposition~\ref{prop:dynamics_of_automorphisms}, 
$$
\image T^+ \subset {\rm Span} F_\Omega(x) = {\rm Span} \{ x\} = x,
$$
and so $\image T^+ = x$. Proposition~\ref{prop:dynamics_of_automorphisms} also implies that $\Pb(\ker T^-) \cap \Omega = \emptyset$ and $x \in \Pb(\ker T^-)$. Notice that $y \notin \Pb(\ker T^-)$ since $(x,y) \subset \Omega$. 

Similarly, we can find a sequence $(\phi_n)_{n \geq 1}$ in $\Gamma$ such that the limits
$$
S^+ = \lim_{n \rightarrow \infty} \phi_n \quad \text{and}  \quad S^- = \lim_{n \rightarrow \infty} \phi_n^{-1}
$$
exist in $\Pb(\End(\Rb^d))$, $\image S^+ = y$, and $x \notin \Pb(\ker S^-)$. 

Fix some $x^\prime \in \image T^-$ and $y^\prime \in \image S^-$. Since $\Gamma$ is strongly irreducible, by Observation~\ref{obs:str_irred_transverse} there exists $h \in \Gamma$ such that 
\begin{enumerate}
\item $h(y^\prime) \notin\Pb( \ker T^+)$, hence $h\left( \image S^-\right) \not\subset \ker T^+$,
\item $hS^-(x) \notin \Pb(\ker T^+)$, 
\item $h(x^\prime) \notin \Pb(\ker S^+)$, hence $h\left(\image T^-\right) \not\subset \ker S^+$,
\item $hT^-(y) \notin \Pb(\ker S^+)$.
\end{enumerate}

Then consider $g_n = \gamma_n \circ h \circ \phi_n^{-1}$. By our choice of $h$, we have $T^+ \circ h \circ S^- \neq 0$ and hence 
$$
T^+ \circ h \circ S^-=\lim_{n \rightarrow \infty} g_n
$$
in $\Pb(\End(\Rb^d))$. Notice that $\image( T^+ \circ h \circ S^-) = \image T^+ = x$ and by our choice of $h$, 
$$
x \notin \Pb(\ker (T^+ \circ h \circ S^-)).
$$
So 
$$
\image( T^+ \circ h \circ S^-)+  \ker (T^+ \circ h \circ S^-) =x+ \ker (T^+ \circ h \circ S^-) = \Rb^d
$$
and hence, by Proposition~\ref{prop:crit_for_prox}, $g_n$ is proximal for $n$ sufficiently large and $\ell_{g_n}^+ \rightarrow x$. 

Similar reasoning shows that $g_n^{-1}$ is proximal for $n$ sufficiently large and $\ell_{g_n}^-=\ell_{g_n^{-1}}^+ \rightarrow y$. 

\end{proof} 

\section{North-South dynamics}\label{sec:NS_dynamics}

In this section we prove a stronger version of Theorem~\ref{thm:biprox_density} for pairs of extreme points in the limit set whose simplicial distance is greater than two.

\begin{theorem}\label{thm:duality} Suppose that $\Omega \subset \Pb(\Rb^d)$ is a properly convex domain and $\Gamma \leq \Aut(\Omega)$ is strongly irreducible. Assume $x,y \in \Lc_\Omega(\Gamma)$ are extreme points of $\Omega$ and $s_{\partial \Omega}(x,y) > 2$. If $A,B \subset \overline{\Omega}$ are neighborhoods of $x,y$, then there exists  $g \in \Gamma$ with 
\begin{align*}
g\left(\overline{\Omega} \setminus B\right) \subset A \text{ and } g^{-1}\left(\overline{\Omega} \setminus A\right) \subset B.
\end{align*}
\end{theorem}

\begin{remark} Theorem~\ref{thm:duality} is an analogue of a result for $\CAT(0)$ spaces, see~\cite[Chapter 3, Theorem 3.4]{B1995}.
\end{remark}

Before proving the theorem we state and prove one corollary. 

\begin{corollary}\label{cor:duality} Suppose that $\Omega \subset \Pb(\Rb^d)$ is an irreducible properly convex domain and $\Gamma \leq \Aut(\Omega)$ is a discrete group acting co-compactly on $\Omega$. Assume $x,y \in \partial \Omega$ are extreme points and $s_{\partial \Omega}(x,y) > 2$. If $A,B \subset \overline{\Omega}$ are neighborhoods of $x,y$, then there exists  $g \in \Gamma$ with 
\begin{align*}
g\left(\overline{\Omega} \setminus B\right) \subset A \text{ and } g^{-1}\left(\overline{\Omega} \setminus A\right) \subset B.
\end{align*}
\end{corollary}

\begin{proof}[Proof of Corollary~\ref{cor:duality}] A result of Vey~\cite[Theorem 5]{V1970} implies that $\Gamma$ is strongly irreducible. Proposition~\ref{prop:dynamics_of_automorphisms_2} implies that $\partial \Omega = \Lc_\Omega(\Gamma)$. So Theorem~\ref{thm:duality} implies the corollary. 
\end{proof} 

We need one lemma.

\begin{lemma}\label{lem:NS} Suppose that $\Omega \subset \Pb(\Rb^d)$ is a properly convex domain, $\gamma \in \Aut(\Omega)$ is bi-proximal, and $s_{\partial \Omega}(\ell_\gamma^+, \ell_\gamma^-) > 2$. If $A,B \subset \overline{\Omega}$ are neighborhoods of $\ell_\gamma^+,\ell_\gamma^-$, then there exists  $N \geq 0$ such that
\begin{align*}
\gamma^n\left(\overline{\Omega} \setminus B\right) \subset A \text{ and } \gamma^{-n}\left(\overline{\Omega} \setminus A\right) \subset B
\end{align*}
for all $n \geq N$. 
\end{lemma}

\begin{proof} Observation~\ref{obs:iterates_of_proximal} implies that
\begin{equation}
\label{eqn:convergence in some proof}
\ell_\gamma^+ = \lim_{n \rightarrow \infty} \gamma^n(x)
\end{equation}
for all $x \in \Pb(\Rb^d) - \Pb(H_g^-)$ and the convergence is locally uniform. 

We claim that 
\begin{align*}
\Pb(H_g^-) \cap \overline{\Omega} = \{ \ell_g^-\}.
\end{align*}
Proposition~\ref{prop:dynamics_of_automorphisms} implies that $\{\ell_g^-\} \subset \Pb(H_g^-) \cap \overline{\Omega}$. Proposition~\ref{prop:dynamics_of_automorphisms} also implies that $\Omega \cap \Pb(H_g^-)=\emptyset$. So if $y \in \Pb(H_g^-) \cap \overline{\Omega}$, then $[y, \ell_g^-] \subset \Pb(H_g^-)\cap \overline{\Omega}$ and hence $[y,\ell_g^-] \subset \partial \Omega$. Then by Theorem~\ref{thm:char_of_rank_one} part (2) we have $y = \ell_g^+$. So $\Pb(H_g^-) \cap \overline{\Omega} \subset \{ \ell_g^-\}$ and the claim is established. 

Then by the locally uniform convergence in Equation~\eqref{eqn:convergence in some proof}, there exists $N_1 > 0$ such that
\begin{align*}
\gamma^n\left(\overline{\Omega} \setminus B\right) \subset A
\end{align*}
for all $n \geq N_1$. 

Repeating the same argument with $\gamma^{-1}$ shows that there exists $N_2 > 0$ such that
\begin{align*}
\gamma^{-n}\left(\overline{\Omega} \setminus A\right) \subset B
\end{align*}
for all $n \geq N_2$. 

Then $N = \max\{N_1, N_2\}$ satisfies the conclusion of the lemma. 
 
\end{proof}

\begin{proof}[Proof of Theorem~\ref{thm:duality}] By Theorem~\ref{thm:biprox_density} there exists a sequence of bi-proximal elements $(g_n)_{n \geq 1}$ in  $\Gamma$ such that 
\begin{align*}
\lim_{n \rightarrow \infty} \ell^+_{g_n} = x \quad \text{and} \quad \lim_{n \rightarrow \infty} \ell^-_{g_n} = y. 
\end{align*}
Since  $s_{\partial \Omega}(x,y) > 2$ we may pass to a tail of $(g_n)_{n \geq 1}$ and assume that 
\begin{align*}
s_{\partial \Omega}(\ell^+_{g_n},\ell^-_{g_n}) > 2
\end{align*} 
for all $n$. 

Next fix $n$ sufficiently large such that $\ell^+_{g_n}  \in A$ and $\ell^-_{g_n}  \in B$. Then by Lemma~\ref{lem:NS} there exists $m \geq 0$ such that 
\begin{align*}
g_n^m\left(\overline{\Omega} \setminus B\right) \subset A \text{ and } g_n^{-m}\left(\overline{\Omega} \setminus A\right) \subset B.
\end{align*}
So $g = g_n^m$ satisfies the theorem. 
\end{proof}

\section{Fixed points and centralizers}\label{sec:Fixed points and centralizers}

In this section we prove the following result connecting the number of boundary fixed points of an element with the size of its centralizer. 

\begin{theorem}\label{thm:fixed_pts_versus_centralizers} Suppose that $\Omega \subset \Pb(\Rb^d)$ is an irreducible properly convex domain and $\Gamma \leq \Aut(\Omega)$ is a discrete group that acts co-compactly on $\Omega$. If $g \in \Gamma$ has infinite order, then the following are equivalent: 
\begin{enumerate}
\item  there exist two distinct points $x,y \in \partial \Omega$ fixed by $g$ with $s_{\partial \Omega}(x,y) < +\infty$, 
\item $g$ fixes at least three points in $\partial \Omega$, 
\item the cyclic group $g^{\Zb}$ has infinite index in its centralizer. 
\end{enumerate}
\end{theorem} 

As a corollary we will obtain the following. 

\begin{corollary}\label{cor:biprox_prop} Suppose that $\Omega \subset \Pb(\Rb^d)$ is an irreducible properly convex domain and $\Gamma \leq \Aut(\Omega)$ is a discrete group that acts co-compactly on $\Omega$. If $g \in \Gamma$ is bi-proximal, then the following are equivalent: 
\begin{enumerate}
\item $[\ell_g^+, \ell_g^-] \subset \partial \Omega$, 
\item $s_{\partial \Omega}(\ell^+_g, \ell^-_g) < +\infty$,
\item $g$ has at least three fixed points in $\partial \Omega$,
\item the cyclic group $g^{\Zb}$ has infinite index in its centralizer.
\end{enumerate}
\end{corollary}

We will first recall some results established in~\cite{IZ2019}, then prove the theorem and corollary.

\subsection{Maximal Abelian subgroups and minimal translation sets}

We have the following description of maximal Abelian subgroups. 

\begin{theorem}[{Islam-Z.~\cite[Theorem 1.6]{IZ2019}}]\label{thm:max_abelian} Suppose that $\Omega \subset \Pb(\Rb^d)$ is a properly convex domain and $\Gamma \leq \Aut(\Omega)$ is a discrete group that acts co-compactly on $\Omega$. If $A \leq \Gamma$ is a maximal Abelian subgroup of $\Gamma$, then there exists a properly embedded simplex $S \subset \Omega$ such that 
\begin{enumerate}
\item $S$ is $A$-invariant, 
\item $A$ acts co-compactly on $S$, and 
\item $A$ fixes each vertex of $S$. 
\end{enumerate}
Moreover, $A$ has a finite index subgroup isomorphic to $\Zb^{\dim(S)}$. 
\end{theorem}

\begin{remark} The above result is a special case of Theorem 1.6 in~\cite{IZ2019} which holds in the more general case when $\Gamma \leq \Aut(\Omega)$ is a naive convex co-compact subgroup.
\end{remark}

\begin{definition} Suppose that $\Omega \subset \Pb(\Rb^d)$ is a properly convex domain and $g \in \Aut(\Omega)$. Define the \emph{minimal translation length of $g$} to be
\begin{align*}
\tau_\Omega(g): = \inf_{x \in \Omega} H_\Omega(x, g (x))
\end{align*}
and the \emph{minimal translation set of $g$} to be
\begin{align*}
\Min_\Omega(g) = \{ x \in\Omega: H_\Omega(g(x),x) = \tau_\Omega(g) \}.
\end{align*}
\end{definition}

Cooper-Long-Tillmann~\cite{CLT2015} showed that the minimal translation length of an element can be determined from its eigenvalues. 

\begin{proposition}\cite[Proposition 2.1]{CLT2015}\label{prop:min_trans_compute} If $\Omega \subset \Pb(\Rb^d)$ is a properly convex domain and $g \in \Aut(\Omega)$, then 
\begin{align*}
\tau_\Omega(g) = \frac{1}{2} \log \frac{\lambda_1(g)}{\lambda_d(g)}.
\end{align*}
\end{proposition}

\begin{remark} Recall, that 
\begin{align*}
\lambda_1(g) \geq \lambda_2(g) \geq \dots \geq \lambda_d(g)
\end{align*}
denote the absolute values of the eigenvalues of some (hence any) lift of $g$ to $\SL_d^{\pm}(\Rb):=\{ h \in \GL_d(\Rb) : \det h = \pm 1\}$. 
\end{remark}

As a consequence of Proposition~\ref{prop:min_trans_compute} we observe the following. 

\begin{observation}\label{obs:min_trans_iterates} If $\Omega \subset \Pb(\Rb^d)$ is a properly convex domain, $p_0 \in \Omega$, and $g \in \Aut(\Omega)$, then 
$$
\lim_{n \rightarrow \infty} \frac{1}{n} H_\Omega(g^n(p_0), p_0) = \tau_\Omega(g). 
$$
\end{observation}

\begin{proof} Proposition~\ref{prop:min_trans_compute} implies that $\tau_\Omega(g^n) = n \tau_\Omega(g)$ and hence 
$$
\liminf_{n \rightarrow \infty} \frac{1}{n} H_\Omega(g^n(p_0), p_0) \geq \tau_\Omega(g). 
$$
For the other inequality, fix $\epsilon > 0$ and $q \in \Omega$ with $H_\Omega(g(q), q) < \tau_\Omega(g)+\epsilon$. Then 
\begin{align*}
\limsup_{n \rightarrow \infty} & \frac{H_\Omega(g^n(p_0), p_0)}{n}  \leq \limsup_{n \rightarrow \infty} \frac{H_\Omega(g^n(q), q)+ 2H_\Omega(p_0, q)}{n} \\
& \leq    \limsup_{n \rightarrow \infty} \frac{H_\Omega(g^n(q), g^{n-1}(q))+ \dots + H_\Omega(g(q), q)+ 2H_\Omega(p_0, q)}{n} \\
& = \limsup_{n \rightarrow \infty} H_\Omega(g(q), q)+ \frac{2H_\Omega(p_0, q)}{n} < \tau_\Omega(g) + \epsilon.
\end{align*}
Since $\epsilon > 0$ was arbitrary, the proof is complete. 
\end{proof}

Next, given a group $G$ and an element $g \in G$, let $C_G(g)$ denote the centralizer of $g$ in $G$. Then given a subset $X \subset G$, define
\begin{align*}
C_G(X)= \cap_{x \in X} C_G(x).
\end{align*}

In~\cite{IZ2019}, the following result about centralizers and minimal translation sets of Abelian groups was established. 

\begin{theorem}[{Islam-Z.~\cite[Theorem 1.10]{IZ2019}}]\label{thm:centralizers_act_cocpctly} Suppose that $\Omega \subset \Pb(\Rb^d)$ is a properly convex domain, $\Gamma \leq \Aut(\Omega)$ is a discrete group that acts co-compactly on $\Omega$, and $A \leq \Gamma$ is an Abelian subgroup. Then
\begin{align*}
\Min_\Omega(A): =\bigcap_{a \in A} \Min_\Omega(a) 
\end{align*}
is non-empty and $C_{\Gamma}(A)$ acts co-compactly on the convex hull of $\Min_\Omega(A)$ in $\Omega$. 
\end{theorem}

\begin{remark} The above result is a special case of Theorem 1.9 in~\cite{IZ2019} which holds in the more general case when $\Gamma \leq \Aut(\Omega)$ is a naive convex co-compact subgroup.
\end{remark}

Finally, we will use the following observations.

\begin{proposition}\label{prop:min_set_inv_simplex} Suppose that $S \subset \Pb(\Rb^d)$ is a simplex. If $g \in \Aut(S)$ fixes every vertex of $S$, then $\Min_S(g) =S$. 
\end{proposition}

\begin{proof} See for instance~\cite[Proposition 7.3]{IZ2019}.\end{proof}

\begin{observation}\label{obs:centralizer of rank one isometry} Suppose that $\Omega \subset \Pb(\Rb^d)$ is a properly convex domain and $\Gamma \leq \Aut(\Omega)$ is a discrete group. If $g \in \Gamma$ is bi-proximal and $(\ell_g^+, \ell_g^-) \subset \Omega$, then $g^{\Zb}$ has finite index in $C_{\Gamma}(g)$. 
\end{observation} 

\begin{proof} First notice that $C_\Gamma(g)$ preserves $(\ell_g^+, \ell_g^-)$. Since $\Aut(\Omega)$ acts properly on $\Omega$ and $\Gamma \leq \Aut(\Omega)$ is discrete, we then see that $C_{\Gamma}(g)$ acts properly on $(\ell_g^+, \ell_g^-)$. Then $g^{\Zb}$ has finite index in $C_{\Gamma}(g)$ since $g^{\Zb}$ acts co-compactly on $(\ell_g^+, \ell_g^-)$.
\end{proof}

\subsection{Proof of Theorem~\ref{thm:fixed_pts_versus_centralizers}} Fix a maximal Abelian subgroup $A \leq \Gamma$ which contains $g$. Then by Theorem~\ref{thm:max_abelian} there exists $S \subset \Omega$ such that
\begin{itemize}
\item $S$ is a properly embedded simplex,
\item $A$ acts co-compactly on $S$, 
\item $A$ fixes every vertex of $S$, and
\item $A$ has a finite index subgroup isomorphic to $\Zb^{\dim(S)}$. 
\end{itemize}
Since $g$ has infinite order, $\dim(S) \geq 1$. 

We consider a number of cases and prove that in each case \textbf{(1)}, \textbf{(2)}, and \textbf{(3)} are either all true or all false. 

\medskip

\noindent \underline{Case 1:} Assume $\dim(S) \geq 2$. Then clearly \textbf{(1)}, \textbf{(2)}, and \textbf{(3)} are all true. 

\medskip

\noindent \underline{Case 2:} Assume $\dim(S) = 1$. Let $v^+, v^-$ be the vertices of $S$ and fix some $p_0 \in S$. Then, after possibly relabelling, we can assume that
$$
\lim_{n \rightarrow \pm \infty} g^n(p_0) = v^\pm.
$$

\noindent \underline{Case 2 (a):} Assume $s_{\partial \Omega}(v^+, v^-) >2$. Then Theorem~\ref{thm:char_of_rank_one} implies that $g$ is a rank one isometry and $v^\pm  = \ell_g^\pm$. Theorem~\ref{thm:char_of_rank_one} also implies that $v^+, v^-$ are the only fixed points of $g$ in $\partial \Omega$ and $s_{\partial \Omega}(v^+, v^-) = \infty$. Hence \textbf{(1)} and \textbf{(2)} are false. Observation~\ref{obs:centralizer of rank one isometry} implies that $g^{\Zb}$ has finite index in $C_{\Gamma}(g)$  and hence \textbf{(3)} is false.

\medskip

\noindent \underline{Case 2 (b):} Assume $s_{\partial \Omega}(v^+, v^-) = 2$. Then, by definition, \textbf{(1)} is true. Fix $y_0 \in \partial \Omega$ such that $[v^+, y_0] \cup [y_0,v^-]$. 

Pick a sequence $n_j \rightarrow \infty$ such that the limits
$$
T^{\pm} : = \lim_{j \rightarrow \infty} g^{\pm n_j}
$$
exist in $\Pb(\End(\Rb^d))$. Then Proposition~\ref{prop:dynamics_of_automorphisms} implies that $v^{\mp} \in \Pb(\ker T^\pm)$ and $\Pb(\ker T^\pm) \cap \Omega = \emptyset$. This implies that $v^\pm \notin \Pb(\ker T^\pm)$ since $(v^+, v^-) \subset \Omega$. Also, $g$ commutes with $T^\pm$ and hence $g\Pb(\ker T^\pm) = \Pb(\ker T^\pm)$.

Passing to a further sequence we can suppose that $g^{\pm n_j}(y_0) \rightarrow y^\pm$. Then 
$$
[v^+, y^\pm] \cup [y^\pm, v^-] \subset \partial \Omega
$$
Then, since $(v^+, v^-) \subset \Omega$, $y^\pm$ must be distinct from $v^+$ and $v^-$. Since $g^{\pm n_j}(x) \rightarrow v^\pm$ for all $x \in \Pb(\Rb^d) \setminus \Pb(\ker T^\pm)$ we must have $y \in \Pb( \ker T^+ \cap \ker T^-)$. Thus the set 
\begin{align*}
C:=\partial \Omega \cap \Pb(\ker T^+ \cap \ker T^-)
\end{align*}
is non-empty. Then $g$ has a fixed point $y \in C$ since $C$ is $g$-invariant, closed, and convex. So $g$ has at least three fixed points in $\partial \Omega$ and  \textbf{(2)} is true.

Recall that  $v^{\mp} \in \Pb(\ker T^\pm)$ and $\Pb(\ker T^\pm) \cap \Omega = \emptyset$, hence 
$$
[v^+, y] \cup [y, v^-] \subset \partial \Omega.
$$
Let $S^\prime$ be the open simplex with vertices $v^+, v^-, y$. Since $(v^+, v^-) \subset  \Omega$ we have $S^\prime \subset \Omega$. In particular 
\begin{align}
\label{eq:comparison}
H_{S^\prime}(p,q) \geq H_\Omega(p,q)
\end{align}
for all $p,q \in S^\prime$. Since $p_0 \in (v^-, v^+) \subset S^\prime \subset \Omega$, Observation~\ref{obs:min_trans_iterates} implies that
$$
\tau_\Omega(g) = \lim_{n \rightarrow \infty} \frac{ H_\Omega(g^n(p_0), p_0)}{n} = \lim_{n \rightarrow \infty} \frac{ H_{S^\prime}(g^n(p_0), p_0)}{n} = \tau_{S^\prime}(g). 
$$
Then by Equation~\eqref{eq:comparison} and Proposition~\ref{prop:min_set_inv_simplex}
\begin{align*}
S^\prime =\Min_{S^\prime}(g) \subset \Min_\Omega(g).
\end{align*}

Now we claim that $g^{\Zb}$ has infinite index in $C_{\Gamma}(g)$. Theorem~\ref{thm:centralizers_act_cocpctly} implies that there is a compact set $K \subset \Omega$ such that 
$$
S^\prime \cup (v^+, v^-) \subset C_{\Gamma}(g)\cdot K.
$$
Further, $g^{\Zb}$ preserves $(v^+, v^-)$, so it is enough to show that 
$$
\sup_{p \in S^\prime} H_\Omega(p, (v^+, v^-)) =\infty.
$$
Fix a sequence $(p_n)_{n \geq 1}$ in $S^\prime$ converging to $y$. Since $(v^+, v^-) \subset \Omega$ and $[v^+, y] \cup [y,v^-] \subset \partial \Omega$, Observation~\ref{obs:faces} implies that the faces $F_\Omega(v^+)$, $F_\Omega(v^-)$, and $F_\Omega(y)$ are all distinct. Then, by the definition of the Hilbert metric, 
$$
\lim_{n \rightarrow \infty} H_\Omega(p_n, (v^+, v^-)) = \infty.
$$
Thus $g^{\Zb}$ has infinite index in $C_{\Gamma}(g)$ and so \textbf{(3)} is true.

\subsection{Proof of Corollary~\ref{cor:biprox_prop}} Theorem~\ref{thm:fixed_pts_versus_centralizers} implies that \textbf{(2)} $\Rightarrow$ \textbf{(3)} $\Leftrightarrow$ \textbf{(4)} and by definition \textbf{(1)} $\Rightarrow$ \textbf{(2)}. Finally, by Observation~\ref{obs:centralizer of rank one isometry}, \textbf{(4)} $\Rightarrow$  \textbf{(1)}.

\section{Rank in the sense of Prasad-Raghunathan}\label{sec:rank_PS}

In this section we consider the rank of a group in the sense of Prasad and Raghunathan~\cite{PR1972}.

\begin{definition}[Prasad-Raghunathan] 
Suppose that $\Gamma$ is an abstract group. For $i \geq 0$, let $A_i(\Gamma) \subset \Gamma$ be the subset of elements whose centralizer contains a free Abelian group of rank at most $i$ as a subgroup of finite index. Next define $r(\Gamma)$ to be the minimal $i \in \{0,1,2,\dots \} \cup \{ \infty\}$ such that there exist $\gamma_1,\dots, \gamma_m \in \Gamma$ with 
\begin{align*}
\Gamma \subset \bigcup_{j=1}^m \gamma_j A_i(\Gamma).
\end{align*}
Then the \emph{Prasad-Raghunathan rank of $\Gamma$} is defined to be
\begin{align*}
{\rm rank}_{{\rm PR}}(\Gamma) : = \sup \left\{ r(\Gamma^*) : \Gamma^* \text{ is a finite index subgroup of } \Gamma \right\}.
\end{align*}
\end{definition}

Prasad and Raghunathan computed the rank of  lattices in semisimple Lie groups which implies the following. 

\begin{theorem}[{Prasad-Raghunathan~\cite[Theorem 3.9]{PR1972}}]\label{thm:PS_uniform_lattices} Suppose that $\Omega \subset \Pb(\Rb^d)$ is an irreducible properly convex domain. If $\Omega$ is symmetric with real rank $r$ and $\Gamma \leq \Aut(\Omega)$ is a discrete group acting co-compactly on $\Omega$, then ${\rm rank}_{{\rm PR}}(\Gamma)=r$. 
\end{theorem}

As a corollary to Selberg's lemma we have the following lower bound on the Prasad-Raghunathan rank. 

\begin{corollary}[To Selberg's Lemma]\label{cor:selberg} If $\Gamma \leq \PGL_d(\Rb)$ is a finitely generated infinite group, then ${\rm rank}_{\rm PR}(\Gamma) \geq 1$. \end{corollary} 

\begin{proof} By Selberg's lemma, there exists a finite index torsion-free subgroup $\Gamma^* \leq \Gamma$. Notice that every element of $A_0(\Gamma^*)$ has finite order and hence $A_0(\Gamma^*)=\{\id\}$. Then, since $\Gamma^*$ is infinite, 
\begin{equation*}
{\rm rank}_{\rm PR}(\Gamma) \geq r(\Gamma^*) \geq 1. \qedhere
\end{equation*}
\end{proof} 

In this section we will show that the existence of a rank one isometry implies that the Prasad-Raghunathan rank is one. 

\begin{proposition}\label{prop:PR_rank_one} Suppose that $\Omega \subset \Pb(\Rb^d)$ is a properly convex domain and $\Gamma \leq \Aut(\Omega)$ is a finitely generated strongly irreducible discrete group. If there exists a bi-proximal element $g \in \Gamma$ with $(\ell_g^+, \ell_g^-) \subset \Omega$, then
\begin{align*}
{\rm rank}_{{\rm PR}}(\Gamma)=1.
\end{align*}
\end{proposition} 

\begin{remark} The proof of Proposition~\ref{prop:PR_rank_one} is a simple modification of Ballmann and Eberlein's~\cite{BE1987} proof of the analogous statement for $\CAT(0)$ groups. 
\end{remark}

The rest of the section is devoted to the proof of Proposition~\ref{prop:PR_rank_one}. So suppose $\Omega \subset \Pb(\Rb^d)$, $\Gamma \leq \Aut(\Omega)$, and $g \in \Gamma$ satisfy the hypothesis of the proposition. By Corollary~\ref{cor:selberg} it is enough to fix a finite index subgroup $\Gamma^* \subset \Gamma$ and  show that $r(\Gamma^*) \leq 1$. Also by replacing $g$ with a sufficiently large power we may assume that $g \in \Gamma^*$.

\begin{lemma}\label{lem:disjointness} Suppose that $x_1, x_2 \in \partial \Omega$ and $(x_1,x_2) \subset \Omega$. If $A, B \subset \partial \Omega$ are open sets with $\overline{A} \cap \overline{B} = \emptyset$, then we can find disjoint neighborhoods $V_1,V_2$ of $x_1,x_2$ such that: for each $\varphi \in \Aut(\Omega)$ at least one of the following occurs
\begin{enumerate}
\item $\varphi(V_1) \cap A =\emptyset$,
\item $\varphi(V_1) \cap B =\emptyset$,
\item $\varphi(V_2) \cap A =\emptyset$,
\item $\varphi(V_2) \cap B =\emptyset$.
\end{enumerate}
\end{lemma}

\begin{proof} The following argument is essentially the proof of Lemma 3.10 in~\cite{BE1987}.

Fix a distance $d_{\Pb}$ on $\Pb(\Rb^d)$ induced by a Riemannian metric. Then for each $n$ and $j=1,2$, let $V_{j,n}$ be a neighborhood of $x_j$ whose diameter with respect to $d_{\Pb}$ is less than $1/n$.

Suppose for a contradiction that the lemma is false. Then for each $n$ there exists $\varphi_n \in \Aut(\Omega)$ such that 
\begin{align}
\label{eq:intersection}
\varphi_n(V_{j,n}) \cap A \neq \emptyset \quad \text{and} \quad \varphi_n(V_{j,n}) \cap B \neq \emptyset
\end{align}
for $j=1,2$. By passing to a subsequence we can suppose that 
\begin{align*}
T : = \lim_{n \rightarrow \infty} \varphi_n
\end{align*}
exists in $\Pb(\End(\Rb^d))$. Then
\begin{align*}
T(u) = \lim_{n \rightarrow \infty} \varphi_n(u)
\end{align*}
for all  $u \in \Pb(\Rb^d) \setminus \Pb(\ker T)$. Moreover, the convergence is uniform on compact subsets of $\Pb(\Rb^d) \setminus \Pb(\ker T)$.

Proposition~\ref{prop:dynamics_of_automorphisms} implies that $\Pb(\ker T) \cap \Omega = \emptyset$. Then, since $(x_1,x_2) \subset \Omega$, it is impossible for both $x_1,x_2$ to be contained in $\Pb(\ker T)$. So after possibly relabelling, we may assume that $x_1 \notin \Pb(\ker T)$. 

By Equation~\eqref{eq:intersection}, there exists sequences $a_n, b_n \in \partial \Omega$ converging to $x_1$ such that $\varphi_n(a_n) \in A$ and $\varphi_n(b_n) \in B$. Then, since $x_1 \notin \Pb(\ker T)$, 
\begin{align*}
T(x_1) = \lim_{n \rightarrow \infty} \varphi_n(a_n) \in \overline{A}
\end{align*}
and 
\begin{align*}
T(x_1) = \lim_{n \rightarrow \infty} \varphi_n(b_n) \in \overline{B}.
\end{align*}
So $T(x_1) \in \overline{A} \cap \overline{B} = \emptyset$ which is a contradiction. 
\end{proof}

\begin{lemma} $r(\Gamma^*) \leq 1$. \end{lemma}

\begin{proof} The following argument is essentially the proof of Theorem 3.1 in~\cite{BE1987}. 

Since $\Gamma$ is strongly irreducible, $\Gamma^*$ is also strongly irreducible. So by Observation~\ref{obs:str_irred_transverse} there exists $\phi \in \Gamma^*$ such that 
\begin{align*}
\phi \ell_g^+, \phi \ell_g^-, \ell_g^+, \ell_g^-
\end{align*}
are all distinct. Then $h := \phi g \phi^{-1}$ is bi-proximal, $\ell^\pm_h = \phi \ell^\pm_g$, and 
\begin{align*}
(\ell_h^+, \ell_h^-) = \phi (\ell_g^+, \ell_g^-) \subset \Omega. 
\end{align*}

Fix open neighborhoods $A,B \subset \partial \Omega$ of $\ell_h^+, \ell_h^-$ such that $\overline{A} \cap \overline{B} = \emptyset$. Then let $V_1, V_2 \subset \partial \Omega$ be neighborhoods of $\ell_g^+, \ell_g^-$ so that $A,B, V_1, V_2$ satisfy Lemma~\ref{lem:disjointness}. 

By further shrinking each $V_j$, we can assume that each $\partial \Omega \setminus V_j$ is homeomorphic to a closed ball. 

Next let $U_1 \subset V_1$ be a closed neighborhood of $\ell_g^+$ such that: if $x \in U_1$ and $y \in \partial \Omega \setminus V_1$, then $s_{\partial \Omega}(x,y) > 2$. Such a choice is possible by Theorem~\ref{thm:char_of_rank_one} part (2). In a similar fashion, let $U_2 \subset V_2$ be a closed neighborhood of $\ell_g^-$ such that: if $x \in U_2$ and $y \in \partial \Omega \setminus V_2$, then $s_{\partial \Omega}(x,y) > 2$. 

By further shrinking each $U_j$, we can assume that each $U_j$ is homeomorphic to a closed ball. 

By Observation~\ref{obs:fixed pts of proximal elements are extreme} each $\ell^\pm_g, \ell^\pm_h$ is an extreme point of $\Omega$. Further, by Theorem~\ref{thm:char_of_rank_one} part (3) 
\begin{align*}
s_{\partial \Omega}(\ell_g^{\pm}, \ell_h^{\pm}) = \infty = s_{\partial \Omega}(\ell_g^{\pm}, \ell_h^{\mp}).
\end{align*}
So by Theorem~\ref{thm:duality} there exists $\varphi_1, \varphi_2, \psi_1,\psi_2 \in \Gamma^*$ such that 
\begin{enumerate}
\item $\varphi_1(\partial \Omega \setminus A) \subset U_1$ and $\varphi_1^{-1}(\partial \Omega \setminus U_1) \subset A$
\item $\psi_1(\partial \Omega \setminus A) \subset U_2$ and $\psi_1^{-1}(\partial \Omega \setminus U_2) \subset A$
\item $\varphi_2(\partial \Omega \setminus B) \subset U_1$ and $\varphi_2^{-1}(\partial \Omega \setminus U_1) \subset B$
\item $\psi_2(\partial \Omega \setminus B) \subset U_2$ and $\psi_2^{-1}(\partial \Omega \setminus U_2) \subset B$.
\end{enumerate}

We claim that 
\begin{align*}
\Gamma^* =  \varphi_1^{-1} A_1(\Gamma^*) \cup \psi_1^{-1} A_1(\Gamma^*) \cup \varphi_2^{-1} A_1(\Gamma^*) \cup \psi_2^{-1} A_1(\Gamma^*). 
\end{align*}
Fix $\gamma \in \Gamma^*$. By construction at least one of the four possibilities in Lemma~\ref{lem:disjointness} must occur. 

\medskip

\noindent \underline{Case 1:} Assume $\gamma(V_1) \cap A = \emptyset$. Then 
\begin{align}
\label{eq:strict_inequality}
\varphi_1\gamma(U_1) \subsetneq \varphi_1\gamma(V_1) \subset  \varphi_1(\partial \Omega \setminus A) \subset U_1
\end{align}
So by the Brouwer fixed point theorem, $\varphi_1\gamma$ has a fixed point in $x \in U_1$ (recall that $U_1$ is homeomorphic to a closed ball). Further, 
\begin{align*}
(\varphi_1\gamma)^{-1}(\partial \Omega \setminus V_1)\subset (\varphi_1\gamma)^{-1}(\partial \Omega \setminus U_1) \subset \gamma^{-1}(A) \subset \partial \Omega \setminus V_1.
\end{align*}
So $\varphi_1\gamma$ also has a fixed point in $y \in \partial \Omega \setminus V_1$. Now by construction $s_{\partial \Omega}(x,y) > 2$. So by Theorem~\ref{thm:char_of_rank_one} part (1) either 
\begin{align*}
\inf_{p \in \Omega} H_\Omega(\varphi_1\gamma(p), p) = 0
\end{align*}
or $\varphi_1 \gamma$ is bi-proximal with 
\begin{align*}
\{x,y\} = \{ \ell_{\varphi_1 \gamma}^+, \ell_{\varphi_1 \gamma}^-\}.
\end{align*}
In the latter case, $(\ell_{\varphi_1 \gamma}^+, \ell_{\varphi_1 \gamma}^-) \subset \Omega$ and so $\varphi_1 \gamma \in A_1(\Gamma)$ by Observation~\ref{obs:centralizer of rank one isometry}. Thus we have reduced to showing that
\begin{align*}
\inf_{p \in \Omega} H_\Omega(\varphi_1\gamma(p),p) > 0.
\end{align*}
Assume for a contradiction that 
\begin{align*}
\inf_{p \in \Omega} H_\Omega(\varphi_1\gamma(p),p) = 0.
\end{align*}
Then by Proposition~\ref{prop:min_trans_compute} we have 
\begin{align*}
\lambda_1(\varphi_1\gamma) = \lambda_2(\varphi_1\gamma) = \cdots = \lambda_d(\varphi_1\gamma).
\end{align*}
Since $x$ and $y$ are eigenlines of $\varphi_1\gamma$ this implies that $\varphi_1\gamma$ fixes every point of the line $(x,y)$. Then, since $\Aut(\Omega)$ acts properly on $\Omega$ and $\Gamma^*$ is discrete,  the group 
\begin{align*}
K = \left\{  (\varphi_1\gamma)^n : n \in \Zb\right\}
\end{align*} 
is finite. So $(\varphi_1\gamma)^N = \id$ for some large $N$. Then Equation~\eqref{eq:strict_inequality} implies that 
\begin{align*}
U_1 = (\varphi_1\gamma)^N(U_1) \subsetneq U_1.
\end{align*}
So we have a contradiction and hence 
\begin{align*}
\inf_{p \in \Omega} H_\Omega(\varphi_1\gamma(p),p) > 0
\end{align*}
and so $\varphi_1 \gamma \in A_1(\Gamma^*)$.
\medskip

\noindent \underline{Case 2:} Assume $\gamma(V_1) \cap B = \emptyset$. Then arguing as in Case 1 shows that $\varphi_2 \gamma \in A_1(\Gamma^*)$.

\medskip

\noindent \underline{Case 3:} Assume $\gamma(V_2) \cap A = \emptyset$. Then arguing as in Case 1 shows that $\psi_1 \gamma \in A_1(\Gamma^*)$.

\medskip

\noindent \underline{Case 4:} Assume $\gamma(V_2) \cap B = \emptyset$. Then arguing as in Case 1 shows that $\psi_2 \gamma \in A_1(\Gamma^*)$.

\medskip 

Since $\gamma \in \Gamma^*$ was arbitrary, we see that 
\begin{align*}
\Gamma^* = \varphi_1^{-1} A_1(\Gamma^*) \cup \psi_1^{-1} A_1(\Gamma^*) \cup \varphi_2^{-1} A_1(\Gamma^*) \cup \psi_2^{-1} A_1(\Gamma^*). 
\end{align*}
Hence $r(\Gamma^*)\leq 1$.
\end{proof}

\section{Proof of Theorem~\ref{thm:characterization}}\label{sec:pf_of_char}

Suppose for the rest of the section that $\Omega \subset \Pb(\Rb^d)$ is an irreducible properly convex domain and $\Gamma \leq \Aut(\Omega)$ is a discrete group that acts co-compactly on $\Omega$. We will show that the following conditions are equivalent: 

\begin{enumerate}
\item $\Omega$ is symmetric with real rank at least two.
\item $\Omega$ has higher rank.
\item \label{char:line_between_ext_pts} The extreme points of $\Omega$ form a closed proper subset of $\partial \Omega$.
\item $[x_1,x_2] \subset \partial \Omega$ for every two extreme points $x_1,x_2 \in \partial \Omega$.
\item $s_{\partial\Omega}(x,y) \leq 2$ for all $x,y \in \partial \Omega$.
\item $s_{\partial\Omega}(x,y) < +\infty$ for all $x,y \in \partial \Omega$.
\item $\Gamma$ has higher rank in the sense of Prasad-Raghunathan. 
\item For every $g \in \Gamma$ with infinite order the cyclic group $g^{\Zb}$ has infinite index in the centralizer $C_\Gamma(g)$ of $g$ in $\Gamma$.
\item Every $g \in \Gamma$ with infinite order has at least three fixed points in $\partial \Omega$.
\item\label{char:line_between_attracting_repelling} $[\ell_g^+, \ell_g^-] \subset \partial \Omega$ for every bi-proximal element $g \in \Gamma$.
\item $s_{\partial\Omega}(\ell_g^+, \ell_g^-) < +\infty$ for every bi-proximal element $g \in \Gamma$.
\item There exists a boundary face $F \subset \partial \Omega$ such that 
\begin{align*}
F \cap \overline{\Ec_\Omega} = \emptyset.
\end{align*}

\end{enumerate}

{\setlength\mathsurround{0pt}

\begin{figure}
\begin{center}
\begin{tikzcd}[row sep=4 em, column sep=4em]
   (4) \arrow[Leftarrow]{rrrr}{Lem.~\ref{lem:10implies4}} \arrow[Rightarrow]{dr}{Defn.} \arrow[Rightarrow]{ddddd}{Lem.~\ref{lem:4implies12}} & & &  & (10) \arrow[Leftrightarrow]{r}{Cor.~\ref{cor:biprox_prop}} \arrow[Leftarrow]{dddd}{Thm.~\ref{prop:PR_rank_one}} &  (11)  \\
    & (6) \arrow[Rightarrow]{r}{Lem~\ref{lem:6implies8}} \arrow[Leftarrow]{d}{Defn.} &   (8) \arrow[Leftrightarrow]{r}{Thm.~\ref{thm:fixed_pts_versus_centralizers}} & (9) \arrow[Rightarrow]{ur}{Cor.~\ref{cor:biprox_prop}} \\
    & (5)\arrow[Leftarrow]{d}{Lem.~\ref{lem:2implies5}}\\
    & (2)\arrow[Leftarrow]{d}{Lem.~\ref{lem:1implies23}}\\
    & (1) \arrow[Leftarrow, swap]{ld}{Thm.~\ref{thm:main_rigidity_result}} \arrow[Rightarrow]{rd}{Lem.~\ref{lem:1implies23}} \arrow[Rightarrow]{rrr}{Thm.~\ref{thm:PS_uniform_lattices}} & & & (7) \\
    (12) \arrow[Leftarrow]{rr}{Defn.} & & (3) 
  \end{tikzcd}
\caption{The proof of Theorem~\ref{thm:characterization}}\label{fig:pf_of_main_thm}
  \end{center}
\end{figure} 
}

We verify all the implications shown in Figure~\ref{fig:pf_of_main_thm}. First notice that the implications $(3) \Rightarrow (12)$, $(4) \Rightarrow (6)$, and $(5) \Rightarrow (6)$ are by definition. The implication $(1) \Rightarrow (7)$ is due to Prasad-Raghunathan, see Theorem~\ref{thm:PS_uniform_lattices} above. Theorem~\ref{prop:PR_rank_one} implies that $(7) \Rightarrow (10)$. Theorem~\ref{thm:fixed_pts_versus_centralizers} implies that $(8) \Leftrightarrow (9)$. Corollary~\ref{cor:biprox_prop} implies that  $(9) \Rightarrow (10)$ and $(10) \Leftrightarrow (11)$. Theorem~\ref{thm:main_rigidity_result}  implies that $(12) \Rightarrow (1)$. The remaining implications in Figure~\ref{fig:pf_of_main_thm} are given as lemmas below. 

\begin{lemma}\label{lem:1implies23} $(1) \Rightarrow (2)$ and $(1) \Rightarrow (3)$. \end{lemma} 

\begin{proof} These implications follow from direct inspection of the short list of irreducible symmetric properly convex domains.
\end{proof}

\begin{lemma}\label{lem:2implies5} $(2) \Rightarrow (5)$. \end{lemma} 

\begin{proof} Suppose $x,y \in \partial \Omega$. If $[x,y] \subset \partial \Omega$, then $s_{\partial\Omega}(x,y) \leq 1$. If $(x,y) \subset \Omega$, then there exists a properly embedded simplex $S \subset \Omega$ with $\dim(S) \geq 2$ and $(x,y) \subset S$. Then 
\begin{align*}
s_{\partial \Omega}(x,y) \leq s_{\partial S}(x,y) \leq 2.
\end{align*}
Since $x,y \in \partial \Omega$ were arbitrary we see that $(5)$ holds. 
\end{proof}

\begin{lemma}\label{lem:4implies12} $(4) \Rightarrow (12)$. \end{lemma} 

\begin{proof} Fix a boundary face $F \subset \partial \Omega$ of maximal dimension. We claim that
\begin{align*}
\overline{\Ec_\Omega} \cap F = \emptyset.
\end{align*}
Otherwise there exists $x \in F$ and a sequence $x_n \in \Ec_\Omega$ such that $x_n \rightarrow x \in F$. Now fix an extreme point $y \in \partial \Omega \setminus \overline{F}$. Then, by hypothesis, $[x_n,y] \subset \partial\Omega$ for all $n$, so $[x,y] \subset \partial \Omega$. 

Fix $z \in (x,y) \subset \partial \Omega$ and let $C$ denote the convex hull of $y$ and $F$. Then by Observation~\ref{obs:faces}
\begin{align*}
\partial \Omega \supset F_\Omega(z) \supset \relint(C).
\end{align*}
Then 
\begin{align*}
\dim F_\Omega(z) > \dim F
\end{align*}
which is a contradiction. So we must have $\overline{\Ec_\Omega} \cap F = \emptyset$ and hence $(12)$ holds.

\end{proof} 

\begin{lemma}\label{lem:6implies8} $(6) \Rightarrow (8)$. \end{lemma} 

\begin{proof} By Theorem~\ref{thm:max_abelian} every infinite order element $g \in \Gamma$ preserves a properly embedded simplex $S \subset \Omega$ with $\dim(S) \geq 1$. Hence $g$ fixes the vertices $v_1,\dots, v_k$ of $S$. By hypothesis $s_{\partial \Omega}(v_1,v_2) <+\infty$ and hence, by Theorem~\ref{thm:fixed_pts_versus_centralizers}, $g^{\Zb}$ has infinite index in the centralizer $C_\Gamma(g)$. 

\end{proof} 

\begin{lemma}\label{lem:10implies4} $(10) \Rightarrow (4)$. \end{lemma} 

\begin{proof} We prove the contrapositive: if there exist extreme points $x,y \in \partial \Omega$ with $(x,y) \subset \Omega$, then there exists a bi-proximal element $g \in \Gamma$ with $(\ell_g^+,\ell_g^-) \subset \Omega$. If such $x,y$ exist, then by Theorem~\ref{thm:biprox_density}  there exist bi-proximal elements $g_n \in \Gamma$ with $\ell^+_{g_n} \rightarrow x$ and $\ell^-_{g_n} \rightarrow y$. Then for $n$ large we must have $(\ell^+_{g_n}, \ell^-_{g_n}) \subset \Omega$. \end{proof}

\bibliographystyle{alpha}
\bibliography{geom}

\begin{thebibliography}{CLT15}

\bibitem[Bal82]{B1982}
Werner Ballmann.
\newblock Axial isometries of manifolds of nonpositive curvature.
\newblock {\em Math. Ann.}, 259(1):131--144, 1982.

\bibitem[Bal85]{Ballmann1985}
Werner Ballmann.
\newblock Nonpositively curved manifolds of higher rank.
\newblock {\em Ann. of Math. (2)}, 122(3):597--609, 1985.

\bibitem[Bal95]{B1995}
Werner Ballmann.
\newblock {\em Lectures on spaces of nonpositive curvature}, volume~25 of {\em
  DMV Seminar}.
\newblock Birkh\"{a}user Verlag, Basel, 1995.
\newblock With an appendix by Misha Brin.

\bibitem[BE87]{BE1987}
Werner Ballmann and Patrick Eberlein.
\newblock Fundamental groups of manifolds of nonpositive curvature.
\newblock {\em J. Differential Geom.}, 25(1):1--22, 1987.

\bibitem[Ben60]{B1960}
Jean-Paul Benz\'{e}cri.
\newblock Sur les vari\'{e}t\'{e}s localement affines et localement
  projectives.
\newblock {\em Bull. Soc. Math. France}, 88:229--332, 1960.

\bibitem[Ben03]{B2003b}
Yves Benoist.
\newblock Convexes divisibles. {II}.
\newblock {\em Duke Math. J.}, 120(1):97--120, 2003.

\bibitem[Ben06]{B2006}
Yves Benoist.
\newblock Convexes divisibles. {IV}. {S}tructure du bord en dimension 3.
\newblock {\em Invent. Math.}, 164(2):249--278, 2006.

\bibitem[Ben08]{B2008}
Yves Benoist.
\newblock A survey on divisible convex sets.
\newblock In {\em Geometry, analysis and topology of discrete groups}, volume~6
  of {\em Adv. Lect. Math. (ALM)}, pages 1--18. Int. Press, Somerville, MA,
  2008.

\bibitem[BK53]{BK1953}
Herbert Busemann and Paul~J. Kelly.
\newblock {\em Projective geometry and projective metrics}.
\newblock Academic Press Inc., New York, N. Y., 1953.

\bibitem[Bor63]{B1963}
Armand Borel.
\newblock Compact {C}lifford-{K}lein forms of symmetric spaces.
\newblock {\em Topology}, 2:111--122, 1963.

\bibitem[BQ16]{BQ2016}
Yves Benoist and Jean-Fran\c{c}ois Quint.
\newblock {\em Random walks on reductive groups}, volume~62 of {\em Ergebnisse
  der Mathematik und ihrer Grenzgebiete. 3. Folge. A Series of Modern Surveys
  in Mathematics [Results in Mathematics and Related Areas. 3rd Series. A
  Series of Modern Surveys in Mathematics]}.
\newblock Springer, Cham, 2016.

\bibitem[BS87a]{BS1987b}
Keith Burns and Ralf Spatzier.
\newblock Manifolds of nonpositive curvature and their buildings.
\newblock {\em Inst. Hautes \'{E}tudes Sci. Publ. Math.}, (65):35--59, 1987.

\bibitem[BS87b]{BS1987a}
Keith Burns and Ralf Spatzier.
\newblock On topological {T}its buildings and their classification.
\newblock {\em Inst. Hautes \'{E}tudes Sci. Publ. Math.}, (65):5--34, 1987.

\bibitem[CLT15]{CLT2015}
D.~Cooper, D.D. Long, and S.~Tillmann.
\newblock On convex projective manifolds and cusps.
\newblock {\em Advances in Mathematics}, 277:181 -- 251, 2015.

\bibitem[dlH93]{dlH1993}
Pierre de~la Harpe.
\newblock On {H}ilbert's metric for simplices.
\newblock In {\em Geometric group theory, {V}ol. 1 ({S}ussex, 1991)}, volume
  181 of {\em London Math. Soc. Lecture Note Ser.}, pages 97--119. Cambridge
  Univ. Press, Cambridge, 1993.

\bibitem[FK94]{FK1994}
Jacques Faraut and Adam Kor\'{a}nyi.
\newblock {\em Analysis on symmetric cones}.
\newblock Oxford Mathematical Monographs. The Clarendon Press, Oxford
  University Press, New York, 1994.
\newblock Oxford Science Publications.

\bibitem[Gol90]{G1990}
William~M. Goldman.
\newblock Convex real projective structures on compact surfaces.
\newblock {\em J. Differential Geom.}, 31(3):791--845, 1990.

\bibitem[{Isl}19]{I2019}
Mitul {Islam}.
\newblock {Rank-One Hilbert Geometries}.
\newblock {\em arXiv e-prints}, Dec 2019.

\bibitem[IZ21]{IZ2019}
Mitul Islam and Andrew Zimmer.
\newblock A flat torus theorem for convex co-compact actions of projective
  linear groups.
\newblock {\em J. Lond. Math. Soc. (2)}, 103(2):470--489, 2021.

\bibitem[Kni98]{K1998}
Gerhard Knieper.
\newblock The uniqueness of the measure of maximal entropy for geodesic flows
  on rank {$1$} manifolds.
\newblock {\em Ann. of Math. (2)}, 148(1):291--314, 1998.

\bibitem[Koe99]{K1999}
Max Koecher.
\newblock {\em The {M}innesota notes on {J}ordan algebras and their
  applications}, volume 1710 of {\em Lecture Notes in Mathematics}.
\newblock Springer-Verlag, Berlin, 1999.
\newblock Edited, annotated and with a preface by Aloys Krieg and Sebastian
  Walcher.

\bibitem[Mar14]{L2014}
Ludovic Marquis.
\newblock Around groups in {H}ilbert geometry.
\newblock In {\em Handbook of {H}ilbert geometry}, volume~22 of {\em IRMA Lect.
  Math. Theor. Phys.}, pages 207--261. Eur. Math. Soc., Z\"{u}rich, 2014.

\bibitem[Nus88]{N1988}
Roger~D. Nussbaum.
\newblock Hilbert's projective metric and iterated nonlinear maps.
\newblock {\em Mem. Amer. Math. Soc.}, 75(391):iv+137, 1988.

\bibitem[PR72]{PR1972}
Gopal Prasad and M.~S. Raghunathan.
\newblock Cartan subgroups and lattices in semi-simple groups.
\newblock {\em Ann. of Math. (2)}, 96:296--317, 1972.

\bibitem[Qui10]{Q2010}
Jean-Fran\c{c}ois Quint.
\newblock Convexes divisibles (d'apr\`es {Y}ves {B}enoist).
\newblock {\em Ast\'{e}risque}, (332):Exp. No. 999, vii, 45--73, 2010.
\newblock S\'{e}minaire Bourbaki. Volume 2008/2009. Expos\'{e}s 997--1011.

\bibitem[Ver14]{V2014}
Constantin Vernicos.
\newblock On the {H}ilbert geometry of convex polytopes.
\newblock In {\em Handbook of {H}ilbert geometry}, volume~22 of {\em IRMA Lect.
  Math. Theor. Phys.}, pages 111--125. Eur. Math. Soc., Z\"{u}rich, 2014.

\bibitem[Vey70]{V1970}
Jacques Vey.
\newblock Sur les automorphismes affines des ouverts convexes saillants.
\newblock {\em Ann. Scuola Norm. Sup. Pisa (3)}, 24:641--665, 1970.

\bibitem[Vin63]{V1963}
\`E.~B. Vinberg.
\newblock The theory of homogeneous convex cones.
\newblock {\em Trudy Moskov. Mat. Ob\v{s}\v{c}.}, 12:303--358, 1963.

\bibitem[Vin65]{V1965}
\`E.~B. Vinberg.
\newblock Structure of the group of automorphisms of a homogeneous convex cone.
\newblock {\em Trudy Moskov. Mat. Ob\v{s}\v{c}.}, 13:56--83, 1965.

\end{thebibliography}

\end{document}